\documentclass[onecolumn]{autart}

\usepackage[latin1]{inputenc}
\usepackage[T1]{fontenc}
\usepackage{amsmath}
\usepackage{amsfonts}
\usepackage{amssymb}
\usepackage{graphicx}
\usepackage{epstopdf}
\usepackage{tikz}
\usepackage{enumerate}
\usepackage{color}
\usepackage{xcolor}
\usepackage{colortbl}
\usepackage{subfig}

\newtheorem{definition}{Definition}
\newtheorem{theorem}{Theorem}

\newtheorem{lemma}{Lemma}

\newtheorem{remark}{Remark}

\newtheorem{example}{Example}

\newtheorem{assumption}{Assumption}
\newenvironment{proof}[1][Proof]{\textbf{#1.} }{\ \rule{0.5em}{0.5em}}
\newcommand{\R}{\mathbb R}

\begin{document}

 \begin{frontmatter}

\title{Leader-following consensus for multi-agent systems with nonlinear dynamics subject to additive bounded disturbances and asynchronously sampled outputs (long version)}

\thanks[footnoteinfo]{This paper was not presented at any IFAC 
meeting. Corresponding author T. M\'enard.}

 \author[tomas]{Tomas M\'enard}\ead{e-mail: tomas.menard@unicaen.fr},
 \author[lias]{Syed Ali Ajwad}\ead{e-mail: syed.ali.ajwad@univ-poitiers.fr},
 \author[emmanuel]{Emmanuel Moulay}\ead{emmanuel.moulay@univ-poitiers.fr},
 \author[lias]{Patrick Coirault}\ead{patrick.coirault@univ-poitiers.fr} 
  and 
 \author[michael]{Michael Defoort}\ead{michael.defoort@uphf.fr}
 
\address[tomas]{LAC (EA 7478), Normandie Universit\'{e}, UNICAEN, 6 bd du Mar\'{e}chal Juin, 14032 Caen Cedex, France.}
 \address[lias]{LIAS (EA 6315), Universit\'{e} de Poitiers, 2 rue Pierre Brousse, 86073 Poitiers Cedex 9, France.}
 \address[emmanuel]{XLIM (UMR CNRS 7252), Universit\'{e} de Poitiers, 11 bd Marie et Pierre Curie, 86073 Poitiers Cedex 9, France.}

 \address[michael]{LAMIH, UMR CNRS 8201, Polytechnic University Hauts-de-France, Valenciennes, 59313 France.}

\begin{abstract}
This paper is concerned with the leader-following consensus problem for a class of Lipschitz nonlinear multi-agent systems with uncertain dynamics, where each agent only transmits its noisy output, at discrete instants and independently from its neighbors. The proposed consensus protocol is based on a continuous-discrete time observer, which provides a continuous time estimation of the state of the neighbors from their discrete-time output measurements, together with a continuous control law. The stability of the multi-agent system is analyzed with a Lyapunov approach and the exponential practical convergence is ensured provided that the tuning parameters and the maximum allowable sampling period satisfy some inequalities. The proposed protocol is simulated on a multi-agent system whose dynamics are ruled by a Chua's oscillator.
\end{abstract}
\end{frontmatter}

\section{Introduction}
The study of Multi-Agent Systems (MAS) has been considered by many researchers during the last decades, due to its important practical applications, such as formation of UAV, attitude synchronization of spacecraft or distributed sensor networks \cite{Olfati-Saber2004}.
MAS are usually characterized by a topology network which reflects the possible ways of communication among agents. A fundamental problem for MAS is to design protocols such that all the agents in the network reach a common value. This problem can be subdivided into two categories, the leaderless consensus and the leader-following consensus. In the leaderless consensus problem, the final common position of the agents cannot be selected. Then, it might be useful to consider a real or virtual leader whose prescribed trajectory has to be followed by all the agents \cite{Hu2012}.\\
Many results have been obtained for MAS whose dynamics are linear, see for instance \cite{Hong2006,Zhou2014}. However, in many practical cases, MAS are governed by more complex dynamics, namely nonlinear dynamics. These nonlinear dynamics usually cannot be neglected in order to obtain more accurate control procedures and objectives. Consensus protocols using the full state information have been considered in \cite{Li2012,Defoort2015} for a fixed topology, in \cite{Wen2014} for time-varying topology or in \cite{Ma2014} for second order dynamics. When the state of the agents is only partially available, that is when only the measured output can be used, it is then necessary to use observers in order to reconstruct the state. Such a strategy has been used  for example in \cite{Hu2014} for a general class of nonlinear MAS and in \cite{Wan2017} for heterogeneous agents.\\
In the aforementioned nonlinear protocols, the considered signals are assumed to be available continuously in real time. But in most applications, it is more desirable or sometimes only possible to transmit the measurements in a discrete way. This may be due to technical constraints or for energy saving \cite{Ploennigs2010,Guo2014}. It is then important to adapt the consensus protocol in order to deal with the sampled signals. A first idea has been to consider time-triggered sampling. Several protocols have been proposed when the state of the agents can be fully measured. Impulsive control for some specific kinds of systems has been considered in \cite{He2017a}. The delay input approach has been used in \cite{Wen2013,Wen2014a,Ding2015} for deterministic sampling period and in \cite{Shen2012,Wan2016,He2017} for stochastic sampling. When only the output is available at discrete instants, then an observer must be designed. A distributed observer protocol with a zero order hold input control has been proposed in \cite{Wan2017a}. In this work, the sampling periods have to be synchronized between the agents.\\
Another approach, which allows aperiodic and asynchronous sampling periods, is event-triggering based consensus protocol. First-order integrators have been considered in \cite{Dimarogonas2012} and general linear systems in \cite{Yue2013,Zhao2019}. Some specific nonlinear classes of systems have also been considered in the literature. MAS with Lure's nonlinear dynamics have been investigated in \cite{Huang2016} and first order nonlinear systems in \cite{Xie2015,Liu2018,Yang2019}. Other classes of nonlinear systems have been treated in \cite{Kaviarasan2018,Liu2018a}. Though providing interesting results, event-triggered schemes involve a more complicated set-up and additional parameters to tune. Indeed, a threshold function has to be considered which dictates the sampling instants for the data transmission. Furthermore, one has to be careful about the Zeno phenomenon which can occur for some schemes \cite{Ding2018}.\\
Most works with discrete signal transmission hold the control input constant between sampling instants. One takes advantage here, of the fact that time-varying control input can be considered. Indeed, only the transmitted signals have to be sampled, not the input. Continuous-discrete time observers, which reconstruct the state in continuous time from discrete-time measurements, have been greatly developed these last years, as in \cite{Farza2014,Bouraoui2015,Hernandez-Gonzalez2016} for different classes of nonlinear systems, and are then used here to tackle the problem of leader-following consensus where only discrete-time outputs are transmitted through the network. It should be noted that these works only consider the observer design, the convergence is obtained by assuming that the input belongs to a bounded set and then cannot be applied directly to the problem considered here. This idea has already been exploited in \cite{Li2017,Phillips2016,Phillips2019} for linear systems, following an hybrid approach, where sufficient conditions for convergence are obtained,  based on LMIs. An high-gain approach has been followed in  \cite{Menard2019,Ajwad2019} for the leaderless and leader-following consensus of MAS with double integrator dynamics. The control part of this high-gain approach is mainly based on \cite{Bedoui2008}, but where only state feedback is considered. These works are then extended here to the leader-following consensus for a class of systems with nonlinear and uncertain dynamics. The main novelty of the paper is the design of a leader-following consensus protocol for a multi agent system, whose agent dynamics belong to a class of uniformly observable multi-output nonlinear systems, where only a part of the state is measured, and each agent's output is transmitted at some discrete instants to its neighbors independently of the other agents. Several features of the proposed approach have to be emphasized. Firstly, only the sampled outputs have to be transmitted through the network, it is not necessary to transmit the inputs. Secondly, the data sent by the agents are not needed to be synchronized, each agent can send its measurements independently from its neighbors, provided that the maximum allowable sampling period is bounded. This allows to reduce the overall bandwidth of the network. Thirdly, the proposed protocol has only three tuning parameters, namely $\bar c,\lambda,\theta$, where $\bar c$ is the coupling force, $\lambda>0$ represents the speed of convergence of the control part and $\theta>0$ represents the speed of convergence of the observer part. Then, the tuning of the proposed scheme is relatively simple and can be adapted easily by the practitioner since the effect of the modification of each parameter has a direct physical meaning. Fourth, the class of nonlinear systems considered here is challenging since it is quite large and it has not yet been considered in the literature for aperiodic and asynchronous sampling periods. A new protocol is thus proposed here to tackle this problem.\\

The paper is organized as follows. Some notations and existing results are recalled in Section \ref{section_notations}. The class of considered MAS is depicted in section \ref{section_model}. The proposed protocol together with a convergence result is reported in Section \ref{section_main_part}. Section \ref{section_example} contains an example illustrating the performances of the proposed protocol. Finally, Section \ref{section_conclusion} concludes the paper.
\section{Preliminaries}\label{section_notations}
In this paper, the following notations will be used. The symbol $\stackrel{\triangle}{=}$ means equal by definition. The set of $n\times n$ real matrices is denoted by $\mathbb R^{n\times n}$. The transpose for real matrices is represented by the superscript $T$. $I_n$ is the identity matrix of dimension $n$, $0_{m\times n}$ is the zero matrix of dimension $m\times n$ and $0_{n}\stackrel{\triangle}{=}0_{n\times n}$. The Kronecker product of matrices $A$ and $B$ is $A\otimes B$. For a symmetric matrix $M$, $\rho_{\max}(M)$ and $\rho_{\min}(M)$ respectively denote the maximum and minimum eigenvalue of $M$. The notation $\text{diag}(w_1,\dots,w_q)$, with $w_i\in\mathbb R^{m\times m}$, $i=1,\dots,q$, $q,m\in\mathbb N$, is used for the diagonal by block matrix with $w_1,\dots,w_q$ on its diagonal. The positive definiteness of a matrix $M$ is denoted $M>0$. The vector of dimension $N\in\mathbb N$ with all entries equal to $1$ is denoted $\mathbf{1}_N$.\\
A directed graph $\mathcal G$ is a pair $(\mathcal V,\mathcal E)$, where $\mathcal V$ is a nonempty finite set of nodes and $\mathcal E\subseteq\mathcal V\times \mathcal V$ is a set of edges, in which an edge is represented by an ordered pair of distinct nodes. For an edge $(i,j)$, node $i$ is called the parent node, node $j$ the child node, and $i$ is a neighbor of $j$. A graph with the property that $(i,j)\in\mathcal E$ implies $(j,i)\in\mathcal E$ is said to be undirected. A path on $\mathcal G$ from node $i_1$ to node $i_l$ is a sequence of ordered edges of the form $(i_k,i_{k+1})$, $k=1,\dots,l-1$. A directed graph has or contains a directed spanning tree if there exists a node called the root, which has no parent node, such that there exists a directed path from this node to every other node in the graph. Suppose that there are $N$ nodes in a graph. The adjacency matrix $\mathcal{A}=(a_{ij})\in\mathbb R^{N\times N}$ is defined by $a_{ii}=0$ and $a_{ij}=1$ if $(j,i)\in\mathcal E$ and $a_{ij}=0$ otherwise. The Laplacian matrix $\mathcal L\in\mathbb R^{N\times N}$ is defined as $\mathcal L_{ii}=\sum_{j\ne i} a_{ij}$, $\mathcal L_{ij}=-a_{ij}$ for $i\ne j$.
\begin{definition}\cite[Def. 1.2 and Th. 2.3-G20 p.133-134]{Berman1994}
 A non singular real matrix $Q\in\mathbb Z^{n\times n}$ is called an $M$-matrix if all its diagonal elements are positive, all of its off-diagonal elements are non positive, and each of its eigenvalues has positive real part.
\end{definition}
\begin{lemma}\label{lemma_m_matrix}\cite[Th. 2.3-H24 p.134]{Berman1994}
Suppose that $Q\in\mathbb Z^{n\times n}$ is an $M$-matrix. Then, there exists a positive vector $\omega=\begin{bmatrix} \omega_1&\dots&\omega_n\end{bmatrix}^T$, such that $\Omega Q+Q^T\Omega>0$, where $\Omega=\text{diag}(\omega_1,\dots,\omega_n)$.
\end{lemma}
\begin{lemma}\label{lemma_continuous_discrete_time}
 Let $n\in\mathbb N$ and $v_i:\R \rightarrow \mathbb R$, $i=1,\dots,n$ be some functions $C^1$ on $(0,+\infty)$ and such that $v_i(t)=0$ for $t<0$, verifying
   \begin{equation}\label{eqn_lemma_continuous_discret_time_000}
  \frac{d}{dt}\left(\sum_{i=1}^n\gamma_iv_i^2(t)\right)\le\sum_{i=1}^n\left(-a_iv_i^2(t)
  +b_i\int_{t-\delta}^t v_i^2(s)ds\right)+k,
 \end{equation}
for all $t\ge0$, where $\gamma_i>0$, $a_i>0$, $b_i\ge 0$, $i=1,\dots,n$, $k\ge0$ and
 \begin{equation}\label{eqn_lemma_continuous_discret_time_001}
 \delta<\min\left(\frac{(\sqrt{2}-1)}{2}\min_{i=1,n}\left(\frac{a_i}{b_i}\right),\frac{1}{\sqrt{2}}\min_{i=1,n}\left(\frac{\gamma_i}{a_i}\right) \right).
\end{equation}
Then, the following inequality holds true
\begin{equation}
 \sum_{i=1}^n\gamma_iv_i^2(t)\le \varsigma e^{-\vartheta t}+\frac{k}{\vartheta},
\end{equation}
with $\vartheta=\frac{1}{2}\min_{i=1,\dots,n}\left(\frac{a_i}{\gamma_i}\right)$ and $\varsigma=\sum_{i=1}^n\gamma_iv_i^2(0)$.
\end{lemma}
\begin{proof}
Let us define $\beta=\min_{i=1,\dots,n}\left(\frac{a_i}{\sqrt{2}\gamma_i}\right)$, $c_i=\frac{b_i}{a_i}\frac{e^{\beta \delta}-1}{\beta}$, $\kappa_i=1-c_i$ and $\kappa=\min_i \kappa_i$. Using the inequality
\begin{equation}
 e^x\le 1+\sqrt{2}x,\quad\forall x\in\left[0,\frac{1}{2}\right],
\end{equation}
leads to 
\begin{align}
 0<c_i\le&\frac{b_i}{a_i}\left(\frac{1+\sqrt{2}\beta\delta-1}{\beta}\right),\text{ since }\beta\delta<\frac{1}{2},\\
 \le&\frac{b_i}{a_i}\sqrt{2}\delta,\\
 \le&\frac{b_i}{a_i}\sqrt{2}\left(\frac{\sqrt{2}-1}{2}\frac{a_i}{b_i}\right),\text{ using eq. \eqref{eqn_lemma_continuous_discret_time_001}},\\
 \le&1-\frac{1}{\sqrt{2}},\label{eqn_lemma_continuous_discret_time_002}
\end{align}
and $\frac{1}{\sqrt{2}}\le\kappa\le1$.\\
Consider the candidate Lyapunov functional
  \begin{equation*}
 W(v_t)=
 \sum_{i=1}^n\left(\gamma_i v_i^2(t)
 +b_i\int_0^{\delta}\int_{t-s}^te^{\kappa\beta(\nu-t+s)} v_i^2(\nu)d\nu ds\right)
\end{equation*}
with $v_t(s)=v(t+s)$, $s\in[-\delta,0]$ and $v=\begin{bmatrix} v_1&\dots&v_n\end{bmatrix}^T$. Then, using the equality $\left[W(v_t)-\sum_i\left(\gamma_i v_i^2(t)\right)\right]=\sum_i\left(b_i \int_0^{\delta}\int_ {t-s}^te^{\kappa\beta(\nu-t+s)}v_i^2(\nu)d\nu ds\right)$, one gets
  \begin{align}
\dot W(v_t)=& \frac{d}{dt}\left(\sum_{i=1}^n\gamma_i v_i^2(t)\right)
 -\kappa\beta\left[W(v_t)-\left(\sum_{i=1}^n\gamma_i v_i^2(t)\right)\right]
 +\int_0^{\delta}\sum_{i=1}^n\left(e^{\kappa\beta s}b_i v_i^2(t)-b_i v_i^2(t-s)\right)ds.
 \end{align}
 Using inequality \eqref{eqn_lemma_continuous_discret_time_000} and the equality $\int_0^{\delta}v_i^2(t-s)ds=\int_{t-\delta}^{t}v_i^2(s)ds$, one obtains
   \begin{align}
\dot W(v_t)\le & -\left(\sum_{i=1}^na_i v_i^2(t)\right)
-\kappa\beta\left[W(v_t)-\left(\sum_{i=1}^n\gamma_i v_i^2(t)\right)\right]+\left(\frac{e^{\kappa\beta\delta}-1}{\kappa\beta}\right)\left(\sum_{i=1}^nb_i v_i^2(t)\right)+k,\\
 \le& -\kappa\beta W(v_t) +k +\sum_{i=1}^na_i\left(-1+\frac{\kappa\beta\gamma_i}{a_i}+\frac{b_i}{a_i}\left(\frac{e^{\kappa\beta\delta}-1}{\kappa\beta}\right)\right)v_i^2(t),\\
 \le&-\kappa\beta W(v_t) +k+\sum_{i=1}^na_i\left(-1+\frac{\kappa}{\sqrt{2}}+c_i\right)v_i^2(t),
\end{align}
 where the latter inequality is obtained by using the fact that $\beta\gamma_i/a_i\le 1/\sqrt{2}$ by definition of $\beta$ and the inequality $\frac{b_i}{a_i}\left(\frac{e^{\kappa\beta\delta}-1}{\kappa\beta}\right)\le \frac{b_i}{a_i}\left(\frac{e^{\beta\delta}-1}{\beta}\right)=c_i$ since the function $x\mapsto \left(\frac{e^{x\delta}{-1}}{x}\right)$ is increasing over $(0,+\infty)$ and $\kappa\in(0,1)$. Further using the inequality $\kappa\le\kappa_i=1-c_i$, $i=1,\dots,n$, gives
\begin{align}
 \dot W(v_t)\le& -\kappa\beta W(v_t)+k
 -\sum_{i=1}^n\left(1-\frac{1}{\sqrt{2}}\right)a_i\kappa v_i^2(t),\nonumber\\
 \le&-\kappa\beta W(v_t)+k.
\end{align}
Finally, since $\kappa\beta\ge \frac{1}{2}\min_i\left(\frac{a_i}{\gamma_i}\right)\stackrel{\triangle}{=}\vartheta$, using the comparison lemma \cite[lemma 3.4 p. 102]{Khalil2002} gives
\begin{align}
 W(v_t)\le& W(v_0)e^{-\vartheta t}+\frac{k}{\vartheta}\le\left(\sum_{i=1}^n\gamma_iv_i^2(0)\right)e^{-\vartheta t}+\frac{k}{\vartheta},
\end{align}
where the latter inequality is obtained by using the fact that $v_i^2(t)=0$ for $t<0$, $i=1,\dots,n$.
\end{proof}
\begin{lemma}\cite[Lemma 4]{Ajwad2019}\label{lemma_technical_inequalities}
 \begin{enumerate}[(i)]
  \item Let $M\in\mathbb R^{n\times n}$ be a symmetric positive definite matrix. Then, one has $x^TMy\le \sqrt{x^TMx}\sqrt{y^TMy}$ for all $x,y\in\mathbb R^n$.
  \item Let $M\in\mathbb R^{n\times n}$ be a symmetric matrix. Then, one has $\rho_{\min}(M)x^Tx\le x^TMx\le\rho_{\max}(M)x^Tx$ for all $x\in\mathbb R^n$.
  \item One has $\sum_{i=1}^n\sqrt{\alpha_i}\le\sqrt{n}\sqrt{\sum_{i=1}^n\alpha_i}$ for all $\alpha_i\ge0$.
  \item Let $A\in\R^{n\times n}$ be a symmetric positive definite matrix and $B\in\R^{m\times m}$ be a symmetric semi-definite matrix, then the following inequality holds
  \begin{equation}
   \rho_{\min}(A)I_n\otimes B\le A\otimes B\le \rho_{\max}(A)I_n\otimes B.
  \end{equation}
 \end{enumerate}
\end{lemma}
\section{Problem statement}\label{section_model}
\subsection{Dynamical model of the agents}
One considers a group of $N$ agents whose dynamics are nonlinear. More precisely, the $i$-th agent dynamics, $i=1,\dots,N$, are given by
 \begin{align}
    \dot x_i^{(1)}(t)&=x_i^{(2)}(t)+\varphi_1\left(t,x_i^{(1)}(t)\right)+\varepsilon_i^{(1)}(t),\label{eqn_general_dynamic}\\
  &~~\vdots\nonumber\\
  \dot x_i^{(q-1)}(t)&=x_i^{(q)}(t)+\varphi_{q-1}\left(t,x_i^{(1)}(t),\dots,x_i^{(q-1)}(t)\right)\nonumber\\
  &\quad+\varepsilon_i^{(q-1)}(t),\nonumber\\
  \dot x_i^{(q)}(t)&=u_i(t)+\varphi_{q}\left(t,x_i^{(1)}(t),\dots,x_i^{(q)}(t)\right)+\varepsilon_i^{(q)}(t),\nonumber\\
  y_i&=x_i^{(1)}+w_i,\nonumber
 \end{align}
where $x_i^{(k)}\in\mathbb R^m$, $k=1,\dots,q$, is the state,  $u_i\in\mathbb R^m$ is the input, $y_i\in\mathbb R^m$ the output, $\varepsilon_i^{(k)}:\R\rightarrow \R^m$, $k=1,\dots,q$, the dynamics uncertainties, $w_i:\R\rightarrow \R^m$ the noise and $\varphi_k:\mathbb R\times \mathbb R^{km}\rightarrow \mathbb R^m$, $q,m\in\mathbb N$ the nonlinearities.\\
System (\ref{eqn_general_dynamic}) is then made up of $q$ blocs, each one of size $m$.\\
Let us denote $x_i=\begin{pmatrix}\left(x_i^{(1)}\right)^T&\cdots&\left(x_i^{(q)}\right)^T \end{pmatrix}^T\in\mathbb R^n$ the state of the $i$-th agent, with $n=qm$, and $\varepsilon_i=\begin{pmatrix}\left(\varepsilon_i^{(1)}\right)^T\dots&\left(\varepsilon_i^{(q)}\right)^T \end{pmatrix}^T$. System (\ref{eqn_general_dynamic}) can be written in the following compact form 
\begin{equation}\label{eqn_general_system_compact_form}
 \begin{cases}
  \dot x_i(t)=Ax_i(t)+\varphi(t,x_i(t))+Bu_i(t)+\varepsilon_i(t),\\
  y_i=Cx_i+w_i,
 \end{cases}
\end{equation}
with $A\in\mathbb R^{n\times n}$, $B\in\mathbb R^{n\times m}$ and $C\in\mathbb R^{m\times n}$ the corresponding matrices and $\varphi=\begin{pmatrix} \varphi_1^T&\dots&\varphi_q^T\end{pmatrix}^T$.
\begin{remark}
 The class of systems considered here is closely related to the class of uniformly observable systems, which has been introduced in \cite{Gauthier1981} for Single Input Single Output (SISO) systems. Indeed, the system is composed of a chain of integrator and the nonlinear part has a triangular structure. While there exists canonical form for SISO systems, no such canonical form exists for Multi Input Multi Output systems. Nevertheless, several classes of uniformly observable systems have been considered in the literature such as in \cite{Farza2004,Farza2011}. The class of systems considered here is the same as in \cite{Farza2004} (where more details are given about possible required change of coordinates) except that the input is supposed to act linearly and without zero dynamics.
\end{remark}
The nonlinear function $\varphi$ is supposed to verify the following assumption. \begin{assumption}\label{assumption_nonlinear_lipschitz}
 The function $\varphi$ is globally Lipschitz in $x$ uniformly with respect to $t$, that is there exists $L_{\varphi}>0$ such that
 \begin{equation}\label{eqn_assumption_lipschiz}
  \Vert \varphi(t,x_1)-\varphi(t,x_2)\Vert\le L_{\varphi}\Vert x_1-x_2\Vert,
 \end{equation}
for all $t\in\mathbb R$ and $x_1,x_2\in\mathbb R^n$.
\end{assumption}
The aim is to design a protocol such that all the agents converge toward a leader, denoted agent $0$, whose dynamics are given by
\begin{align}\label{eqn_general_leader}
 \dot x_0(t)&=A x_0(t)+\varphi(t,x_0(t))+\varepsilon_0(t),\\
 y_0&=C x_0+w_0.\nonumber
\end{align}
\begin{remark}
It should be noted that the uncertainties on the dynamics of the leader can also represent an unknown input. Indeed, given the structure of the dynamics (see equation \eqref{eqn_general_dynamic}), the uncertainty on the dynamics of $x_0^{(q)}$ can be decomposed as an unknown leader input $u_0$ on one part and some perturbations on the dynamics on the other part.
\end{remark}
The dynamics uncertainties and noises of the agents and the leader are supposed to be uniformly bounded.
\begin{assumption}\label{assumption_boundedness_uncertainties_noise}
 There exist constants $\delta_{\varepsilon}^{1},\dots,\delta_{\varepsilon}^q\ge0$ and $\delta_w\ge 0$ such that 
 \begin{align}
  \Vert \varepsilon_i^{(k)}(t)\Vert&\le \delta_{\varepsilon}^k,\,\forall t\ge0,\, i=0,\dots,N,\, k=1,\dots,q,\\
  \Vert w_i(t)\Vert&\le \delta_{w},\, \forall t\ge 0,\, i=0,\dots,q.
 \end{align}
 \end{assumption}
Due to the uncertainties, the agents cannot converge exactly toward the leader but, instead, in some ball around the leader. The problem to be solved here is defined more precisely below. 
\begin{definition}
The leader-following consensus problem is said to be exponentially practically solved if there exist $\alpha,\beta>0$ and $\gamma\ge 0$ such that $\Vert x_i(t)-x_0(t)\Vert\le \alpha e^{-\beta t}+\gamma$, for all $i=1,\dots,N$.
\end{definition}

\subsection{Communication constraints}
The communication topology between followers $i=1,\dots,N$ is denoted $\mathcal G$ and its adjacency and  Laplacian matrices $\mathcal A=(a_{ij})$ and $\mathcal L$ respectively.\\
The output of the leader is supposed to be transmitted only at some agents. More precisely, one defines $\mathcal D=\text{diag}(d_1,\dots,d_N)$, where $d_i=1$ if agent $i$ has access to the output of the leader and $d_i=0$ otherwise.\\
One defines the general index $\nu_{ij}$ for $i=1,\dots,N,j=0,\dots,N$ as $\nu_{ij}=1$ if agent $i$ receives the output of agent $j$ and $0$ otherwise. This means that if $\nu_{ij}=1$, then the output of agent $j$ is transmitted to agent $i$ at time instants $\left(t_k^{i,j}\right)_{k\in\mathbb N}$. These sampling instants are supposed to verify
 \begin{equation}\label{eqn_condition_time_sequence_transmission}
 t_0^{i,j}<t_1^{i,j}<\dots<t_k^{i,j}<\dots\text{ and }\tau_m<\left\vert t_{k+1}^{i,j}-t_k^{i,j}\right\vert<\tau_M,
\end{equation}
for all $k\in\mathbb N$ and $\tau_m,\tau_M>0$. Note that the lower bound $\tau_m$ is just considered in order to explicitly avoid the Zeno phenomena. It is not a restrictive bound since it can be taken as small as desired.\\
\begin{example}\label{example_1}
Consider a MAS composed of a leader (denoted $0$) and two agents (denoted $1$ and $2$) such as depicted on Figure \ref{fig_graphe_exemple}.\\
Agent 1 receives the transmitted output $y_0$ of the leader at time instants $t_k^{1,0}$ for $k\in\mathbb N$. This is used to reconstruct the state of the leader in continuous time. Agent 1 also uses its own output at time instants $t_k^{1,1}$ for $k\in\mathbb N$, to reconstruct its own state in continuous time. Similarly, Agent 2 receives the output of agent 1 at time instants $t_k^{2,1}$ and uses its own output at time instants $t_k^{2,2}$.\\
Note that the time sequences $\left(t_k^{1,0}\right)_{k\in\mathbb N}$,  $\left(t_k^{1,1}\right)_{k\in\mathbb N}$, $\left(t_k^{2,1}\right)_{k\in\mathbb N}$,  $\left(t_k^{2,2}\right)_{k\in\mathbb N}$ can be chosen freely, and in particular independently from each other, as long as they verify (\ref{eqn_condition_time_sequence_transmission}).
 \begin{figure}[!ht]
 \begin{center}
 \scalebox{1}{
  \begin{tikzpicture} [,>=stealth,line width=0.8pt]
\draw (0,4)    node (1) [circle,draw,fill=gray!20] {\Large$\bf0$};   
\draw (3,4)    node (2) [circle,draw,fill=gray!20] {\Large$\bf 1$};
\draw (3,2)    node (3) [circle,draw,fill=gray!20] {\Large$\bf 2$};
\draw[fill=gray!00] (1)--(2) node[midway,above]{\footnotesize{$y_0\left(t_k^{1,0}\right)$}};
\draw[fill=gray!00] (2)--(3) node[midway,right]{\footnotesize{$y_1\left(t_k^{2,1}\right)$}};
\draw [->] (1) to (2); 
\draw [->] (2) to (3); 
 \draw [->] (2) .. controls ++(1,-0.5) and ++(1,0.5) .. (2); 
 \draw [->] (3) .. controls ++(1,-0.5) and ++(1,0.5) .. (3);
 \draw (4.6,4) node {\footnotesize$y_1\left(t_k^{1,1}\right)$};
 \draw (4.6,2) node {\footnotesize$y_2\left(t_k^{2,2}\right)$};
\end{tikzpicture}}
\end{center}
\caption{Data transmission along the directed graph $\mathcal G$}
\label{fig_graphe_exemple}
 \end{figure}
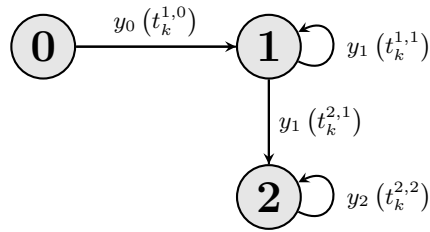
\end{example}
One denotes  $\tilde{\mathcal G}$ the digraph corresponding to all the agents $i=1,\dots,N$ together with the leader. The Laplacian matrix $\tilde{\mathcal L}$ of $\tilde {\mathcal G}$ is given by
\begin{equation}
 \tilde{\mathcal L}=\begin{pmatrix}0&0_{1\times N}\\-\tilde d&\mathcal L+\mathcal D\end{pmatrix}\stackrel{\triangle}{=}\begin{pmatrix}0&0_{1\times N}\\-\tilde d&\mathcal H\end{pmatrix},
\end{equation}
where $\tilde d=\begin{pmatrix} d_1&\dots&d_N\end{pmatrix}^T$.\\
One needs the following result.
\begin{lemma}\label{lemma_nonsingular_m_matrix}\cite[Lemma 9]{Song2012}
 The matrix $\mathcal H$ is a non-singular M-matrix if and only if the communication topology $\tilde{\mathcal G}$ has a directed spanning tree.
\end{lemma}
\begin{assumption}\label{assumption_topology}
 The communication topology $\tilde{\mathcal G}$ between the agents and the leader contains a directed spanning tree.
\end{assumption}
\begin{remark}\label{remark_existence_omega}
 If assumption \ref{assumption_topology} holds true, then according to lemma  \ref{lemma_nonsingular_m_matrix} and lemma \ref{lemma_m_matrix}, there exists a positive vector $\omega=\begin{pmatrix} \omega_1&\dots\omega_N \end{pmatrix}$ such that $\Omega\mathcal H+\mathcal H^T\Omega>0$ where $\Omega=diag(\omega_1,\dots,\omega_N)$. Then, in the rest of the paper, one denotes
 \begin{align}
  \varrho&=\rho_{\min}\left( \Omega\mathcal H+\mathcal H^T\Omega \right),\label{eqn_def_varrho}\\
  \omega_{\min}&=\min(\omega_1,\dots,\omega_N),\label{eqn_def_omega_min}\\
  \omega_{\max}&=\max(\omega_1,\dots,\omega_N).\label{eqn_def_omega_max}
 \end{align}
\end{remark}

\section{Main result}\label{section_main_part}
\subsection{Consensus protocol}
The proposed protocol is given, for $i=1,\dots,N$, by
\begin{align}
 u_i(t)&=d_i\bar c K^c\Gamma_{\lambda}(\hat x_{i,0}(t)-\hat x_{i,i}(t))
+\bar cK^c\Gamma_{\lambda}\sum_{j=1}^N a_{ij}(\hat x_{i,j}(t)-\hat x_{i,i}(t)),\quad\forall t\ge0,\label{eqn_protocol_consensus_control}
\end{align}
where $\hat x_{i,j}$ is the estimate of $x_j$ by agent $i$ and is given, for $t\in\big[t_{k}^{i,j},t_{k+1}^{i,j}\big)$, $k\in\mathbb N$, by
\begin{equation} 
 \dot{\hat x}_{i,j}(t)=A\hat x_{i,j}(t)+\varphi(t,\hat x_{i,j}(t))-\theta\Delta_{\theta}^{-1}K^oz_{i,j}(t),\label{eqn_protocol_consensus_observer}  
\end{equation}
  with 
  \begin{equation}\label{eqn_def_z_ij}
   z_{i,j}(t)=e^{-\theta K_1^o(t-t_k^{i,j})}\left(C\hat x_{i,j}\left(t_k^{i,j}\right)-y_j\left(t_k^{i,j}\right)\right),
  \end{equation}
 where $\bar c,\lambda,\theta>0$ are the tuning parameters, and
\begin{align}
\Gamma_{\lambda}&=\text{diag}\left(\lambda^qI_m,\lambda^{q-1}I_m,\dots,\lambda I_m\right),\\
\Delta_{\theta}&=\text{diag}\left(I_m,\frac{1}{\theta} I_m,\dots,\frac{1}{\theta^{q-1}}I_m\right),\\
K^o&=P^{-1}C^T, \qquad K^c=B^TQ,
\end{align}
where $q$ is the number of blocs of system (\ref{eqn_general_dynamic}), $P,Q\in\mathbb R^{n\times n}$ are the symmetric positive definite solutions of the following matrix equalities
\begin{align}
 P+PA+A^TP&=C^TC,\label{eqn_def_P}\\
 Q+QA+A^TQ&=QBB^TQ,\label{eqn_def_Q}
\end{align}
(see \cite{Bedoui2008} for more details).
\begin{remark}
Given the structure by block of $A,B,C$, one can show directly (see \cite{Farza2011} section 2.2 and \cite{Bedoui2008} section 3  for more details) that $K^o$ and $K^c$ can be written as follows
\begin{align}
K^o&=\begin{pmatrix}K^o_1I_m\dots K^o_qI_m\end{pmatrix}^T,\label{eqn_explicit_expression_K_o1}\\
K^c&=\begin{pmatrix}K^c_1 I_m\dots K^c_qI_m\end{pmatrix},\label{eqn_explicit_expression_K_c1}
\end{align}
where $K^o_1,\dots,K^o_q,K^c_1,\dots,K^c_q\in\mathbb R$ are equal to 
\begin{align}
 K^o_i&= \begin{pmatrix} q\\i \end{pmatrix},\quad K^c_i=\begin{pmatrix} q\\q-i+1\end{pmatrix},\quad i=1,\dots,q,\label{eqn_explicit_expression_K_o_K_c}
\end{align}
and $\begin{pmatrix}n\\k\end{pmatrix}$ are the binomial coefficients.
\end{remark}
\begin{remark}\label{remark_output_error_prediction}
 The term $z_{i,j}(t)$, defined in (\ref{eqn_def_z_ij}) corresponds to a prediction of the output error $\left(C\hat x_{i,j}(t)-y_j(t)\right)$ on each interval $\big[t_k^{i,j},t_{k+1}^{i,j}\big)$ (see \cite{Farza2014} section III.B for more details).
\end{remark}
\begin{remark}
The proposed structure presents some advantages when dealing with delays and dropouts. Indeed, each control input is fed by the corresponding local observers. Then, if the measurements are time-stamped, as soon as the measurement is received, even if delayed, the estimation of the current state can be provided by making the observer computation run faster to compensate.
\end{remark}
\begin{example}
Consider the same topology as in Example \ref{example_1} and assume that the dynamics of the agents are ruled by a second order system, more precisely $\dot x_i^{(1)}=x_i^{(2)}$, $\dot x_i^{(2)}=u_i+\varphi_2(x_i)$, $y_i=x_i^{(1)}\in\R$. It means here that there are two blocks ($q=2$) and the dimension of each agent is equal to $n=2$. One further has $d_1=1$, $d_2=0$, $a_{11}=a_{12}=a_{22}=0$ and $a_{21}=1$.\\
Then, system (\ref{eqn_general_system_compact_form}) is characterized by $A=\begin{pmatrix}0&1\\0&0\end{pmatrix}$, $B=\begin{pmatrix}0\\1\end{pmatrix}$ and $C=\begin{pmatrix}1&0\end{pmatrix}$. Furthermore, the solution of equations (\ref{eqn_def_P}) and (\ref{eqn_def_Q}) are equal to $P=\begin{pmatrix}1&-1\\-1&2 \end{pmatrix}$ and $Q=\begin{pmatrix}1&1\\1&2 \end{pmatrix}$ and the gains are given by $K^o=\begin{pmatrix}2\\1 \end{pmatrix}$, $K^c=\begin{pmatrix}1&2 \end{pmatrix}$, $\Gamma_{\lambda}=\begin{pmatrix}\lambda^2&0\\0&\lambda \end{pmatrix}$, $\Delta_{\theta}=\begin{pmatrix} 1&0\\0&\frac{1}{\theta}\end{pmatrix}$.\\
Given the topology of the considered MAS:
\begin{itemize}
 \item Agent $1$ has to reconstruct the state of the leader and its own state. Then agent $1$ has to run two observers:
 \begin{align*}
  \dot{\hat x}_{1,0}(t)&=A\hat x_{1,0}(t)+\begin{pmatrix}0\\\varphi_2(\hat x_{1,0}(t)) \end{pmatrix}
  -\theta\Delta_{\theta}^{-1}K^o e^{-2\theta\left(t-t_k^{1,0}\right)}\left(\hat x_{1,0}^{(1)}\left(t_k^{1,0}\right)-y_0\left(t_k^{1,0}\right)\right),
  \quad\text{for }t\in\Big[t_k^{1,0},t_{k+1}^{1,0}\Big),\\
  \dot{\hat x}_{1,1}(t)&=A\hat x_{1,1}(t)+\begin{pmatrix}0\\\varphi_2(\hat x_{1,1}(t)) \end{pmatrix}
  -\theta \Delta_{\theta}^{-1}K^oe^{-2\theta\left(t-t_k^{1,1}\right)}\left(\hat x_{1,1}^{(1)}\left(t_k^{1,1}\right)-y_1\left(t_k^{1,1}\right)\right),
  \quad\text{for }t\in\Big[t_k^{1,1},t_{k+1}^{1,1}\Big),
 \end{align*}
where $\hat x_{1,0}$ and $\hat x_{1,1}$ are the estimates of $x_0$ and $x_1$ respectively. The input of agent $1$ is then given by:
\begin{equation*}
 u_1(t)=2\bar c\lambda^2\left(\hat x_{1,0}^{(1)}(t)-\hat x_{1,1}^{(1)}(t)\right)+\bar c\lambda\left(\hat x_{1,0}^{(2)}(t)-\hat x_{1,1}^{(2)}(t)\right).
\end{equation*}
 \item Agent $2$ has to reconstruct the state of agent 1 and its own state. Similarly as for agent 1, agent $2$ has to run two observers: one whose state $\hat x_{2,1}$ is an estimate of $x_1$ and one whose state $\hat x_{2,2}$ is an estimate of $x_2$.The input of agent $2$ is then given by:
\begin{equation*}
 u_2(t)=2\bar c\lambda^2\left(\hat x_{2,1}^{(1)}(t)-\hat x_{2,2}^{(1)}(t)\right)+\bar c\lambda\left(\hat x_{2,1}^{(2)}(t)-\hat x_{2,2}^{(2)}(t)\right).
\end{equation*}
\end{itemize}
It should be noted that the only parameters which must be tuned are $\bar c,\lambda$ and $\theta$ since it is clear from (\ref{eqn_explicit_expression_K_o1})-(\ref{eqn_explicit_expression_K_o_K_c}) that $K^o$ and $K^c$ only depend on the structure of the system (i.e. the number of blocks $q$ and the size of each block $m$).
\end{example}

\subsection{Convergence result}
The convergence of the proposed consensus protocol is now analyzed.
\begin{theorem}\label{theorem_convergence_result}
 Consider the MAS (\ref{eqn_general_dynamic})-(\ref{eqn_general_leader}) subject to Assumptions \ref{assumption_nonlinear_lipschitz}, \ref{assumption_boundedness_uncertainties_noise} and \ref{assumption_topology} and the consensus protocol (\ref{eqn_protocol_consensus_control})-(\ref{eqn_protocol_consensus_observer})-(\ref{eqn_def_z_ij}). If the tuning parameters $\lambda,\theta,\bar c\ge1$ are chosen such that
 \begin{align*}
 \bar c&\ge c^*=\max\left(\frac{\omega_{\max}}{\varrho},1\right),\\
 \lambda&\ge\lambda^*=24L_{\varphi}\sqrt{n}\max\left(\frac{\sqrt{\rho_M^{Q\omega}}}{\sqrt{\rho_m^{Q\omega}}}, \frac{\sqrt{\rho_{M}^P}}{\sqrt{\rho_{m}^P}}\right),\\
  \theta&\ge\lambda\bar c^2\xi^*,
  \end{align*}
  with $\xi^*\stackrel{\triangle}{=}36\Vert K^c\Vert^2(N+1)^3h_{\max}^2\max\left(\frac{\sqrt{\rho_{M}^P}}{\sqrt{\rho_{m}^P}},\frac{\rho_{M}^P}{\rho_m^{Q\omega}}, \frac{\rho_M^{Q\omega}}{\rho_{m}^P} \right)$, $\rho_m^P=\rho_{\min}(P)$,  $\rho_M^P=\rho_{\max}(P)$, $\rho_m^{Q\omega}=\omega_{\min}\rho_{\min}(Q)$, $\rho_M^{Q\omega}=\omega_{\max}\rho_{\max}(Q)$, $h_{\max}=\max_{ij}\vert\mathcal H_{ij}\vert$, $\varrho,\omega_{\min},\omega_{\max}$ respectively defined by \eqref{eqn_def_varrho}-\eqref{eqn_def_omega_min}-\eqref{eqn_def_omega_max}, and the upper bound on the sampling periods $\tau_M$ verifies 
 \begin{equation}
  \tau_M<\frac{\sigma^*}{\bar c(\theta+L_{\varphi})},
 \end{equation}
with $ \sigma^*=\left(\frac{\sqrt{2}-1}{8}\right)\frac{\min(\sqrt{\omega_{\min}},\sqrt{\rho_m^P})}{\Vert K^o\Vert(N+1)^{\frac{3}{2}}h_{\max}\sqrt{\rho_M^P}}$, then the consensus error $\Vert x_i-x_0\Vert$ verifies
\begin{align}\label{eqn_over_valuation_error_theorem}
 \Vert x_i-x_0\Vert\le& \chi_1\theta^{q-1} e^{-\frac{\lambda}{8}t}+\chi_2\lambda^{-1}\theta^q\left(\delta_w+\tau_M\delta_{\varepsilon}^1\right)
 +\chi_3\lambda^{-1}\left(\sum_{k=1}^q\theta^{q-k}\delta_{\varepsilon}^k\right),
\end{align}
where $\chi_1,\chi_2,\chi_3\ge0$ are independent of the tuning parameters $\lambda,\theta,\bar c$ and given by
\begin{align*}
 \chi_1&=\frac{\sqrt{\rho_M^{Q\omega}}}{\sqrt{\rho_m^{Q\omega}}}\sum_{i=1}^N\Vert x_i(0)-x_0(0)\Vert
 +\frac{\sqrt{\rho_M^P}}{\sqrt{\rho_m^{Q\omega}}}\sum_{i=1}^N\sum_{j=0}^N\Vert \hat x_{i,j}(0)-x_j(0)\Vert,\\
 \chi_2&=\frac{8\sqrt{\rho_M^P}(N+1)\Vert K^o\Vert}{\sqrt{\rho_m^{Q\omega}}},\\
 \chi_3&=\frac{8\left(2N\sqrt{\rho_M^{Q\omega}}+(N+1)\sqrt{\rho_M^P}\right)}{\sqrt{\rho_m^{Q\omega}}}.
\end{align*}
\end{theorem}
\begin{remark}
 Equation \eqref{eqn_over_valuation_error_theorem} states that even if the exponential practical consensus is reached, not all the uncertainties effect on the tracking error can be lowered through the tuning of $\lambda$ and $\theta$. Indeed, it is only possible for the uncertainties appearing on the last block of system \eqref{eqn_general_dynamic} (that is only $\varepsilon_i^{(q)}$ are non zero) and if $q\ge 2$. This is done by increasing $\lambda$ and $\theta$, since the corresponding term $\chi_3\lambda^{-1}\delta_{\varepsilon}^q$ goes to zero as $\lambda$ increases.\\
 In particular, if the leader has an unknown but bounded non zero input, then this input can be seen as a dynamic uncertainty on the last block dynamics $x_0^{(q)}$. Therefore, the tracking error can be set as low as desired by increasing $\lambda$ and $\theta$.
\end{remark}
\begin{remark}
 Theorem \ref{theorem_convergence_result} only provides sufficient conditions for the proposed consensus scheme (\ref{eqn_protocol_consensus_control})-(\ref{eqn_protocol_consensus_observer}). Indeed, the consensus may be obtained even if the bounds given by Theorem \ref{theorem_convergence_result} are not respected. Conservativeness is a general drawback when considering general classes of nonlinear systems with a Lyapunov approach. Nevertheless, the convergence analysis gives some useful hints for the tuning of the control parameters. In fact, the bounds $\sigma^*,c^*,\lambda^*,\xi^*$ only depend on the structure of the system (number of blocks $q$ and size of each block $m$) and the topology of the network $\tilde{\mathcal G}$. Given the inequalities in Theorem \ref{theorem_convergence_result}, the coupling force should be tuned first as for a classical consensus protocol. Then, $\lambda$ (representing the speed of convergence of the control part) should be chosen high enough to dominate the nonlinear Lipschitz term. The parameter $\theta$ of the observer part should be higher than the parameter of the control part $\lambda$. This is due to the fact that the input is not transmitted through the network, only the output is transmitted. Finally, the maximum bound on the sampling periods will have to be chosen small if the Lipschitz constant or if $\theta$ takes high values. This corresponds to a kind of Shannon condition: if the system is fast, the sampling periods have to be small.\\ 
 Furthermore, given the bounds on the consensus error given by equation \eqref{eqn_over_valuation_error_theorem}, it can be seen that increasing $\lambda$ and $\theta$ will decrease the effect of some uncertainties (those on the dynamics of $x_i^{(q)}$) but it will increase the effect of the noise, then a trade-off has to be considered between lowering the effect of some uncertainties and amplifying the noise.
\end{remark}
\begin{proof}
The proof of Theorem \ref{theorem_convergence_result} is split into three steps. First, new coordinates are considered in \emph{Step 1}. Then, in \emph{Step 2}, some candidate Lyapunov functions are defined and over-valuations of their derivatives are obtained. Finally, it is shown in \emph{Step 3}, that, if the different inequalities of Theorem \ref{theorem_convergence_result} are verified, then Lemma \ref{lemma_continuous_discrete_time} can be applied and the leader-following consensus problem is solved.\\
\noindent\emph{Step 1.} Let us define $e_i=\Gamma_{\lambda}(x_i-x_0)$ and $\bar e_{i,j}=\Delta_{\theta}(\hat x_{i,j}-x_j)$. Then, using the equalities $\Gamma_{\lambda}A\Gamma_{\lambda}^{-1}=\lambda A$, $\Gamma_{\lambda}B=\lambda B$, $\Delta_{\theta} A\Delta_{\theta}^{-1}=\theta A$, $\Delta_{\theta}B=\frac{1}{\theta^{q-1}}B$, $C\Delta_{\theta}^{-1}=C$ and the notation $\tilde{\varphi}(t,x,y)\stackrel{\triangle}{=}\varphi(t,x)-\varphi(t,y)$, one gets
\begin{align}
 \dot e_i&=\lambda A e_i+\Gamma_{\lambda}\left(\tilde{\varphi}(t,x_i,x_0)+\varepsilon_i-\varepsilon_0\right)+\lambda Bu_i,\label{eqn_proof_e_i}\\
 \dot{\bar e}_{i,j}&=\theta A\bar e_{i,j}+\Delta_{\theta}\tilde{\varphi}(t,\hat x_{i,j},x_j)-\frac{1}{\theta^{q-1}}B u_j
 -\theta K^o z_{i,j} -\Delta_{\theta}\varepsilon_j,\label{eqn_proof_bar_e_ij_01}\\
 &=\theta(A-K^oC)\bar e_{i,j}+\Delta_{\theta}\tilde{\varphi}(t,\hat x_{i,j},x_j)
 -\frac{1}{\theta^{q-1}}B u_j
 -\theta K^o(z_{i,j}-C\bar e_{i,j})
 -\Delta_{\theta}\varepsilon_j,
\end{align}
for $i=1,\dots,N$ and $j=0,\dots,N$, where $z_{i,j}$ is defined by (\ref{eqn_def_z_ij}) and the inputs are given by
\begin{align}
 &u_0=0_{m\times 1},\\
 &u_j= \bar cK^c\left(d_j\Gamma_{\lambda}\Delta_{\theta}^{-1}\bar e_{j,0}-\sum_{k=1}^N\mathcal H_{jk}(e_k+\Gamma_{\lambda}\Delta_{\theta}^{-1}\bar e_{j,k})\right),\nonumber
\end{align}
for $j=1,\dots,N$. Further denoting $\eta^c=\begin{bmatrix} e_1^T&\dots&e_N^T\end{bmatrix}^T$, the dynamics of the $e_i$ can be written in compact form as follows
\begin{equation}
 \dot{\eta}^c=\lambda(I_N\otimes A)\eta^c-\bar c\lambda[\mathcal H\otimes (BK^c)]\eta^c+\Phi_{\lambda}+\Psi_{\lambda,\varepsilon}+\Theta,\nonumber
 \end{equation}
with 
 $ \Phi_{\lambda}=\begin{bmatrix}\Gamma_{\lambda} \tilde{\varphi}(t, x_1,x_0)\\
  \vdots\\
  \Gamma_{\lambda}\tilde{\varphi}(t,x_N,x_0) \end{bmatrix}$
,$\,
 \Psi_{\lambda,\varepsilon}=\begin{bmatrix}
                             \Gamma_{\lambda}(\varepsilon_1-\varepsilon_0)\\
                             \vdots\\
                             \Gamma_{\lambda}(\varepsilon_N-\varepsilon_0)
                            \end{bmatrix},
                            $
and
$ \Theta=\begin{bmatrix}\bar c\lambda BK^c\Gamma_{\lambda}\Delta_{\theta}^{-1}\left(d_1\bar e_{1,0}-\sum_{k=1}^N\mathcal H_{1k}\bar e_{1,k}\right)\\\vdots\\ \bar c\lambda BK^c\Gamma_{\lambda}\Delta_{\theta}^{-1}\left(d_N\bar e_{N,0}-\sum_{k=1}^N\mathcal H_{Nk}\bar e_{N,k}\right)\end{bmatrix}.$

\noindent\emph{Step 2.} One now introduces the candidate Lyapunov functions, one that depends on the control error $\eta^c$ and one that depends on the observer errors $\bar e_{i,j}$.\\
One thus define first $\bar V_c(\eta^c)=(\eta^c)^T[\Omega\otimes Q]\eta^c$, where  $\Omega$ is defined in remark \ref{remark_existence_omega}.\\
Then, for the observer error part, one considers
$\bar V_o(\eta^o)=\sum_{i=1}^N\sum_{j=0}^N\nu_{ij}V_o(\bar e_{i,j})$, with $V_o(\bar e_{i,j})=\bar e_{i,j}^T P\bar e_{i,j}$, where the vector $\eta^o$ contains all the $\bar e_{i,j}$ such that $\nu_{ij}=1$.\\
 One has
\begin{align}
 \dot{\bar V}_c(\eta^c)&=\lambda(\eta^c)^T(\Omega\otimes[A^TQ+QA])\eta^c
 -\bar c\lambda (\eta^c)^T[(\mathcal H^T\Omega)\otimes ((BK^c)^TQ)]\eta^c
  -\bar c\lambda(\eta^c)^T[(\Omega\mathcal H)\otimes (QBK^c)](\eta^c)\nonumber\\
&\quad+2(\eta^c)^T[\Omega\otimes Q]\Phi_{\lambda}+2(\eta^c)^T[\Omega\otimes Q]\Theta
+2(\eta^c)^T[\Omega\otimes Q]\Psi_{\lambda,\varepsilon},\\
 \dot V_o(\bar e_{i,j})&=\theta\bar e_{i,j}^T[(A-K^oC)^TP+P(A-K^oC)]\bar e_{i,j},
 +2\bar e_{i,j}^TP\Delta_{\theta}\tilde{\varphi}(t,\hat x_{i,j},x_j)\nonumber\\
 &\quad -\frac{2}{\theta^{q-1}}\bar e_{i,j}^TPBu_j+2\theta \bar e_{i,j}^TPK^o(C\bar e_{i,j}-z_{i,j}) -2\bar e_{i,j}^TP\Delta_{\theta}\varepsilon_j
\end{align}
for all $i,j$ such that $\nu_{ij}=1$.\\
\emph{Step 2.1 Over-valuation of $\dot {\bar V}_c(\eta^c)$}\\
 Using the equality $(BK^c)^TQ=QBK^c=QBB^TQ$ (obtained with the definition of $Q$ in \eqref{eqn_def_Q}) and Lemma \ref{lemma_technical_inequalities} (iv), one gets
\begin{align}
 (\mathcal H^T\Omega)\otimes((BK^c)^TQ)+(\Omega \mathcal H)\otimes(QBK^c)&=[\mathcal H^T\Omega+\Omega\mathcal H]\otimes [QBB^TQ],\\
 &\ge\frac{\varrho}{\omega_{\max}}\Omega\otimes[QBB^TQ],
\end{align}
with $\varrho$ and $\omega_{\max}$ defined by \eqref{eqn_def_varrho} and \eqref{eqn_def_omega_max} respectively. Taking $\bar c\ge c^*\stackrel{\triangle}{=}\max\{\omega_{\max}/\varrho,1\}$ and using Lemma \ref{lemma_technical_inequalities} (i) and (ii) leads to 
\begin{align}\label{eqn_proof_step_2.1_eq4}
 \dot{\bar V}_c&\le-\lambda\bar V_c
 +2\sqrt{\omega_{\max}\rho_{\max}(Q)}\sqrt{\bar V_c}\left(\Vert \Phi_{\lambda}\Vert+\Vert\Psi_{\lambda,\varepsilon}\Vert+\Vert \Theta\Vert\right).
\end{align}
One now derives over-valuations of the terms $\Vert \Theta\Vert$, $\Vert \Phi_{\lambda}\Vert$ and $\Vert\Psi_{\lambda,\varepsilon}\Vert$.\\
 Using the definition of $d_i,\mathcal H$ and $\nu_{ij}$, and the fact that $\Vert B\Vert=1$ yields
\begin{equation}
 \Vert\Theta\Vert \le \bar c\lambda\Vert K^c\Vert\Vert\Gamma_{\lambda}\Delta_{\theta}^{-1}\Vert h_{\max} \sum_{i=1,j=0}^N\nu_{ij}\Vert \bar e_{i,j}\Vert,\label{eqn_proof_step_2.1_eq1}
 \end{equation}
where $h_{\max}=\max_{i,j}\vert\mathcal H_{ij}\vert$. Further using Lemma \ref{lemma_technical_inequalities} (ii) and (iii), one obtains
 \begin{equation}
 \Vert\Theta\Vert \le \frac{\bar c\lambda\Vert K^c\Vert\,\Vert\Gamma_{\lambda}\Delta_{\theta}^{-1}\Vert (N+1)h_{\max}}{\sqrt{\rho_{\min}(P)}} \sqrt{\bar V_o}. \label{eqn_proof_step_2.1_eq2}
 \end{equation}
 High-gain techniques, such as in \cite{Farza2004}, with $\lambda\ge1$ gives
 \begin{equation}
 \Vert \Phi_{\lambda}\Vert \le  \frac{\sqrt{n}L_{\varphi}}{\sqrt{\omega_{\min}\rho_{\min}(Q)}}\sqrt{\bar V_c}.\label{eqn_proof_step_2.1_eq3}
\end{equation}
For the over-valuation of $\Vert\Psi_{\lambda,\varepsilon}\Vert$, one has
\begin{align}
 \Vert\Psi_{\lambda,\varepsilon}\Vert&\le \sum_{i=1}^N\Vert \Gamma_{\lambda}(\varepsilon_i-\varepsilon_0)\Vert
 \le \sum_{i=1}^N\left(\Vert\Gamma_{\lambda}\varepsilon_i\Vert+\Vert\Gamma_{\lambda}\varepsilon_0\Vert\right),\nonumber\\
 &\le \sum_{i=1}^N\sum_{k=1}^q\lambda^{q-k+1}\left(\Vert\varepsilon_i^{(k)}\Vert+\Vert\varepsilon_0^{(k)}\Vert\right),\\
 &\le \sum_{i=1}^N\sum_{k=1}^q2\delta_{\varepsilon}^k\lambda^{q-k+1}\\
 &\le 2N\sum_{k=1}^q\delta_{\varepsilon}^k\lambda^{q-k+1}.\label{eqn_proof_step_2.1_eq5}
\end{align}
Finally, using inequalities (\ref{eqn_proof_step_2.1_eq4}), (\ref{eqn_proof_step_2.1_eq2}), (\ref{eqn_proof_step_2.1_eq3}) and \eqref{eqn_proof_step_2.1_eq5}, one gets
\begin{align}
 \dot{\bar V}_c&\le -\lambda\bar V_c+2k_1\bar V_c+2\bar ck_2\lambda\Vert \Gamma_{\lambda}\Delta_{\theta}^{-1}\Vert\sqrt{\bar V_c}\sqrt{\bar V_o}
 +2k_3\sqrt{\bar V_c}\left(\sum_{k=1}^q\lambda^{q-k+1}\delta_{\varepsilon}^k\right),\label{ineq_bar_V_c}
\end{align}
with 
\begin{align}
 k_1=&L_{\varphi}\sqrt{n}\sqrt{\frac{\omega_{\max}\rho_{\max}(Q)}{\omega_{\min}\rho_{\min}(Q)}},\label{eqn_def_k1}\\
 k_2=&(N+1)\Vert K^c\Vert h_{\max}\sqrt{\frac{\omega_{\max}\rho_{\max}(Q)}{\rho_{\min}(P)}}, \label{eqn_def_k2}\\
 k_3=&2N\sqrt{\omega_{\max}\rho_{\max}(Q)}.\label{eqn_def_k3}
\end{align}
\emph{Step 2.2 Over-valuation of $\dot V_o(\bar e_{i,j})$}\\
Using the definition of $P$ in equation (\ref{eqn_def_P}), Lemma \ref{lemma_technical_inequalities} (i) and (ii) and the fact that $\Vert B\Vert=1$ give
\begin{align}\label{eqn_proof_step_2.2_eq0}
 \dot V_o\le&-\theta V_o+2\sqrt{\rho_{\max}(P)}\sqrt{V_o}\Vert \Delta_{\theta}\tilde{\varphi}(t,\hat x_{i,j},x_j)\Vert 
 +2\sqrt{\rho_{\max}(P)}\sqrt{V_o}\left(\frac{\Vert u_j\Vert}{\theta^{q-1}}+\theta\Vert K^o\Vert\Vert C\bar e_{i,j}-z_{i,j}\Vert\right)\nonumber\\
 & +2\sqrt{\rho_{\max}(P)}\sqrt{V_o}\left(\Vert\Delta_{\theta}\varepsilon_j\Vert\right).  
\end{align}
One now derives over-valuations of the terms  $\Vert u_j\Vert$, $\Vert \Delta_{\theta}\tilde{\varphi}(t,\hat x_{i,j},x_j)\Vert$, $\Vert C\bar e_{i,j}-z_{i,j}\Vert$ and $\Vert \Delta_{\theta}\varepsilon_j\Vert$ .\\
 Using the definition of $\nu_{ij}$ and $\mathcal H$, one gets 
 \begin{align}
 \Vert u_j\Vert&\le\bar c\Vert K^c\Vert h_{\max} \Bigg(\sum_{k=1}^N \Vert e_k\Vert+\Vert\Gamma_{\lambda}\Delta_{\theta}^{-1}\Vert \sum_{k=0}^N\nu_{jk}\left\Vert \bar e_{j,k}\right\Vert\Bigg).\label{eqn_proof_step_2.2_eq2}
\end{align}
Further applying Lemma \ref{lemma_technical_inequalities} (ii) and (iii) gives
\begin{align}
 \Vert u_j\Vert&\le\bar c\Vert K^c\Vert\sqrt{N+1} h_{\max}\Bigg(\frac{\sqrt{\bar V_c}}{\sqrt{\rho_{\min}(Q)\omega_{\min}}}+\frac{\Vert \Gamma_{\lambda}\Delta_{\theta}^{-1}\Vert\sqrt{\bar V_o}}{\sqrt{\rho_{\min}(P)}}\Bigg).\label{eqn_proof_step_2.2_eq3}
 \end{align}
 Using high-gain techniques such as in \cite{Farza2004}, with $\theta\ge1$, yields
\begin{equation}
 \Vert \Delta_{\theta} \tilde{\varphi}(t,\hat x_{i,j},x_j)\Vert\le\frac{\sqrt{n}L_{\varphi}}{\sqrt{\rho_{\min}(P)}}\sqrt{ V_o(\bar e_{i,j})}.\label{eqn_proof_step_2.2_eq1}
 \end{equation}
Concerning the over-valuation of $\Vert C\bar e_{i,j}-z_{i,j}\Vert$, using equations \eqref{eqn_def_z_ij} and \eqref{eqn_proof_bar_e_ij_01} yields
  \begin{align}
   &\frac{d}{dt}\left(C\bar e_{i,j}(t)- z_{i,j}(t)\right)=\bigg(\theta\bar e_{i,j}^{(2)}-\theta K_1^o z_{i,j}-\varepsilon_j^{(1)}+\varphi_1\left(t,\hat x_{i,j}^{(1)}\right)-\varphi_1\left(t,x_{j}^{(1)}\right)\bigg)-\left(-\theta K_1^o z_{i,j}\right)\text{, if }q\ge2,\label{eqn_proof_0011}\\
   &\frac{d}{dt}\left(C\bar e_{i,j}(t)- z_{i,j}(t)\right)=\bigg(-\theta K_1^o z_{i,j}-u_j-\varepsilon_j^{(1)} +\varphi_1\left(t,\hat x_{i,j}^{(1)}\right)-\varphi_1\left(t,x_{j}^{(1)}\right)\bigg)-\left(-\theta K_1^o z_{i,j}\right)\text{, if }q=1.\label{eqn_proof_0012}
  \end{align}
  for all $t\in\mathbb R, t\ne t_k^{i,j}$, $k\in\mathbb N$.\\
  \noindent Then, denoting $\kappa_{i,j}(t)=\max\left\{t_k^{i,j}\bigg\vert t_k^{i,j}<t,k\in\mathbb N\right\}$ the last instant when agent $j$ transmitted its measurement to agent $i$ at time $t$ and using the fact that $C\bar e_{i,j}(\kappa_{i,j}(t))-z_{i,j}(\kappa_{i,j}(t))=-w_j(\kappa_{i,j}(t))$ for all $t\ge0$, integrating \eqref{eqn_proof_0011} from $\kappa_{i,j}(t)$ to $t$ gives, for $q\ge2$:
 \begin{align}
 &C\bar e_{i,j}(t)-z_{i,j}(t)=-w_j(\kappa_{i,j}(t))+\int_{\kappa_{i,j}(t)}^t\theta\bar e_{i,j}^{(2)}(s)+\varphi_1\left(t,\hat x_{i,j}^{(1)}(s)\right)-\varphi_1\left(t,x_j^{(1)}(s)\right)-\varepsilon_j^{(1)}(s)ds,\label{eqn_proof_step_2.2_eq4}
 \end{align}
 and integrating \eqref{eqn_proof_0012} from $\kappa_{i,j}(t)$ to $t$ gives, for $q=1$:
 \begin{align}
  &C\bar e_{i,j}(t)-z_{i,j}(t)=-w_j(\kappa_{i,j}(t))+\int_{\kappa_{i,j}(t)}^t u_j(s)+\varphi_1\left(t,\hat x_{i,j}^{(1)}(s)\right)-\varphi_1\left(t,x_j^{(1)}(s)\right)-\varepsilon_j^{(1)}(s)ds.\label{eqn_proof_step_2.2_eq9}
 \end{align}
 Further using Assumptions \ref{assumption_nonlinear_lipschitz} and \ref{assumption_boundedness_uncertainties_noise}, one obtains, if $q\ge2$:
 \begin{align}
\Vert C\bar e_{i,j}-z_{i,j}\Vert \le& \delta_{w}+(\theta+L_{\varphi})\int_{\kappa_{i,j}(t)}^t\Vert \bar e_{i,j}(s)\Vert ds +(t-\kappa_{i,j}(t))\delta_{\varepsilon}^1,\label{eqn_proof_step_2.2_eq5}
\end{align}
and, if $q=1$, using \eqref{eqn_proof_step_2.2_eq3} and the fact that $\Vert \Gamma_{\lambda}\Delta_{\theta}^{-1}\Vert=\lambda$ give
\begin{align}
 \Vert C\bar e_{i,j}-z_{i,j}\Vert &\le \delta_{w}+\bar c  \Vert K^c\Vert \sqrt{N+1}h_{\max}\int_{\kappa_{i,j}(t)}^t \frac{\bar V_c(s)}{\sqrt{\omega_{\min}\rho_{\min}(Q)}}  ds+\bar c  \Vert K^c\Vert \sqrt{N+1}h_{\max}\int_{\kappa_{i,j}(t)}^t \lambda\frac{\sqrt{\bar V_o(s)}}{\sqrt{\rho_{\min}(P)}}ds\nonumber\\
 &\quad+L_{\varphi}\int_{\kappa_{i,j}(t)}^t\Vert \bar e_{i,j}(s)\Vert ds+(t-\kappa_{i,j}(t))\delta_{\varepsilon}^{1}.
\end{align}
Applying Lemma \ref{lemma_technical_inequalities} (ii) and the fact that $t-\kappa_{i,j}(t)\le \tau_M$, by definition of $\kappa_{i,j}$ and since $\left\vert t_{k+1}^{i,j}-t_k^{i,j}\right\vert<\tau_M$ for all $k\in\mathbb N$, one gets, if $q\ge 2$:
\begin{align}
 \Vert C\bar e_{i,j}-z_{i,j}\Vert\le& \delta_{w}+\tau_M\delta_{\varepsilon}^{1}+\frac{(\theta+L_{\varphi})}{\sqrt{\rho_{\min}(P)}}\int_{t-\tau_M}^t\sqrt{\bar V_o(\eta^o(s))}ds,\label{eqn_proof_step_2.2_eq6}
\end{align}
and if $q=1$:
\begin{align}
\Vert C\bar e_{i,j}-z_{i,j}\Vert&\le\delta_{w}+\tau_M\delta_{\varepsilon}^{1}+\left(\frac{\bar c  \Vert K^c\Vert \sqrt{N+1}h_{\max}\lambda+L_{\varphi}}{\sqrt{\rho_{\min}(P)}}\right) \int_{t-\tau_M}^t \sqrt{\bar V^o(s)}ds\nonumber\\
&\quad +\left(\frac{\bar c  \Vert K^c\Vert \sqrt{N+1}h_{\max}}{\sqrt{\omega_{\min}\rho_{\min}(Q)}}\right) \int_{t-\tau_M}^t \sqrt{\bar V^c(s)}ds.\label{eqn_proof_step_2.2_eq7}
\end{align}
Furthermore, similarly to the over-valuation of $\Vert \Psi_{\lambda,\varepsilon}\Vert$, one gets
\begin{align}
 \Vert \Delta_{\theta}\varepsilon_j\Vert\le& \sum_{k=1}^q\frac{1}{\theta^{k-1}}\delta_{\varepsilon}^k.\label{eqn_proof_step_2.2_eq8}
\end{align}
Finally, using inequalities (\ref{eqn_proof_step_2.2_eq0}), (\ref{eqn_proof_step_2.2_eq3}),  (\ref{eqn_proof_step_2.2_eq1}), (\ref{eqn_proof_step_2.2_eq6}), (\ref{eqn_proof_step_2.2_eq7}) and (\ref{eqn_proof_step_2.2_eq8})  yields
 \begin{align}\label{eqn_proof_inq_V_o}
 \dot V_o(\bar e_{i,j})&\le-\theta V_o(\bar e_{i,j})+2k_4V_o(\bar e_{i,j})+\frac{2 k_5}{(N+1)}\sqrt{V_o(\bar e_{i,j})}\int_{t-\tau_M}^t\sqrt{\bar V_o(\eta^o(s))}ds\nonumber\\
  &\quad+\frac{2k_6}{(N+1)}\sqrt{V_o(\bar e_{i,j})}\int_{t-\tau_M}^t\sqrt{\bar V_c(\eta^c(s))}ds+\frac{2\bar c\sqrt{V_o(\bar e_{i,j})}}{\theta^{q-1}(N+1)}\left(k_7\sqrt{\bar V_c}+k_8\Vert \Gamma_{\lambda}\Delta_{\theta}^{-1}\Vert \sqrt{\bar V_o}\right)\nonumber\\
 &\quad+\frac{2k_9}{(N+1)}\sqrt{V_o(\bar e_{i,j})}\left(\sum_{k=1}^q\frac{\delta_{\varepsilon}^k}{\theta^{k-1}}\right)+\frac{2\theta k_{10}}{(N+1)}\sqrt{V_o(\bar e_{i,j})}\left( \delta_w+\tau_M\delta_{\varepsilon}^{1}\right),
\end{align}
with 
\begin{align}
 k_4=&L_{\varphi}\sqrt{n}\sqrt{\frac{\rho_{\max}(P)}{\rho_{\min}(P)}},\label{eqn_def_k4}\\
k_5=&\theta(N+1)\Vert K^o\Vert(L_{\varphi}+\theta)\sqrt{\frac{\rho_{\max}(P)}{\rho_{\min}(P)}}\text{, if }q\ge2,\nonumber\\
 =&\theta(N+1)\Vert K^o\Vert\left(\bar c  \sqrt{N+1}h_{\max}\lambda+L_{\varphi}\right)\sqrt{\frac{\rho_{\max}(P)}{\rho_{\min}(P)}} \text{, if }q=1,\label{eqn_def_k5}\\
k_6=&0 \text{, if }q\ge2,\nonumber\\
=& \frac{\bar c \theta  (N+1)^{\frac{3}{2}}h_{\max}}{\sqrt{\omega_{\min}}} \text{, if }q=1,\label{eqn_def_k6}\\
k_7=& \Vert K^c\Vert(N+1)^{\frac{3}{2}}h_{\max}\sqrt{\frac{\rho_{\max}(P)}{\rho_{\min}(Q)\omega_{\min}}},\label{eqn_def_k7}\\
k_8=&\Vert K^c\Vert(N+1)^{\frac{3}{2}}h_{\max}\sqrt{\frac{\rho_{\max}(P)}{\rho_{\min}(P)}},\label{eqn_def_k8}\\
k_9=&\sqrt{\rho_{\max}(P)}(N+1),\label{eqn_def_k9}\\
k_{10}=&\sqrt{\rho_{\max}(P)}\Vert K^o\Vert(N+1).\label{eqn_def_k10}
\end{align}
\emph{Step 2.3 Over-valuation of $\dot{\bar V}_o(\eta^o)$}\\
Using the definition of $\bar V_o$, inequality (\ref{eqn_proof_inq_V_o}) and Lemma \ref{lemma_technical_inequalities} (iii)  give 
\begin{align}\label{ineq_bar_V_o}
 \dot{\bar V}_o(\eta^o)\le& -\theta \bar V_o(\eta^o)+2k_4\bar V_o(\eta^o)+2 k_5\sqrt{\bar V_o(\eta^o)}\int_{t-\tau_M}^t\sqrt{\bar V_o(\eta^o(s))}ds+2k_6\sqrt{\bar V_o(\eta^o)}\int_{t-\tau_M}^t\sqrt{\bar V_c(\eta^c(s))}ds\\
 &+\frac{2\bar ck_7}{\theta^{q-1}}\sqrt{\bar V_o(\eta^o)}\sqrt{\bar V_c(\eta^c)}+\frac{2\bar ck_8\Vert \Gamma_{\lambda}\Delta_{\theta}^{-1}\Vert}{\theta^{q-1}} \bar V_o(\eta^o)+2k_9 \sqrt{\bar V_o(\eta^o)}\left(\sum_{k=1}^q\frac{\delta_{\varepsilon}^k}{\theta^{k-1}}\right)+2\theta k_{10}\sqrt{\bar V_o(\eta^o)}\left(\delta_w+\tau_M\delta_{\varepsilon}^{1}\right).\nonumber
 \end{align}
\emph{Step 3.} 
From inequalities (\ref{ineq_bar_V_c}) and (\ref{ineq_bar_V_o}), one directly obtains
\begin{align*}
 \frac{d}{dt}\left(\sqrt{\bar V_c}\right)&\le-\frac{\lambda}{2}\sqrt{\bar V_c}+k_1\sqrt{\bar V_c}+\bar c k_2\lambda\Vert \Gamma_{\lambda}\Delta_{\theta}^{-1}\Vert\sqrt{\bar V_o}+k_3\left(\sum_{k=1}^q\delta_{\varepsilon}^k\lambda^{q-k+1}\right),\\
 \frac{d}{dt}\left(\sqrt{\bar V_o}\right)&\le-\frac{\theta}{2}\sqrt{\bar V_o}+k_4\sqrt{\bar V_o}+\frac{\bar ck_7}{\theta^{q-1}}\sqrt{\bar V_c}+\frac{\bar c k_8\Vert \Gamma_{\lambda}\Delta_{\theta}^{-1}\Vert}{\theta^{q-1}}\sqrt{\bar V_o}+k_5\int_{t-\tau_M}^t\sqrt{\bar V_o}ds\\
 &\quad +k_6\int_{t-\tau_M}^t\sqrt{\bar V_c}ds +k_9 \left(\sum_{k=1}^q\frac{\delta_{\varepsilon}^k}{\theta^{k-1}}\right)+\theta k_{10}\left(\delta_w+\tau_M\delta_{\varepsilon}^{1}\right).\nonumber
\end{align*}
Taking $\theta=\xi\lambda$ with $\xi\ge1$ leads to $\Vert\Gamma_{\lambda}\Delta_{\theta}^{-1}\Vert=\frac{\theta^{q}}{\xi}$ and
\begin{align}
  \frac{d}{dt}\left(\xi^{\frac{3}{2}}\sqrt{{\bar V}_c}+\theta^q\sqrt{{\bar V}_o}\right)&\le-\frac{\xi^{\frac{3}{2}}\lambda}{4}\sqrt{{\bar V}_c}-\frac{\theta^{q+1}}{4}\sqrt{{\bar V}_o}-\frac{\xi^{\frac{3}{2}}\lambda}{4}\left(1-\frac{4k_1}{\lambda}-\frac{4 \bar ck_7}{\xi^{\frac{1}{2}}}\right)\sqrt{{\bar V}_c} \nonumber\\
&\quad -\frac{\theta^{q+1}}{4}\left(1-\frac{4\bar ck_2}{\xi^{\frac{1}{2}}}-\frac{4k_4}{\theta}-\frac{4 \bar ck_8}{\xi}\right)\sqrt{{\bar V}_o}\nonumber\\
&\quad+ \theta^q k_5\int_{t-\tau_M}^t\sqrt{\bar V_o}ds+ \theta^q k_6 \int_{t-\tau_M}^t\sqrt{\bar V_c}ds\nonumber\\
&\quad+k_3\left(\sum_{k=1}^q\xi^{\frac{3}{2}}\lambda^{q-k+1}\delta_{\varepsilon}^k\right)+k_9 \left(\sum_{k=1}^q\theta^{q-k+1}\delta_{\varepsilon}^k\right)+\theta^{q+1} k_{10}\left( \delta_w+\tau_M\delta_{\varepsilon}^{1}\right).
\end{align}
Then, since $\bar c\ge c^*\ge1$, if the inequalities 
 $\lambda\ge \lambda^*$ and 
$ \xi\ge \bar c^2\xi^*$ hold true,
with $\lambda^*\stackrel{\triangle}{=}24L_{\varphi}\sqrt{n}\max\left(\frac{\sqrt{\rho_M^{Q\omega}}}{\sqrt{\rho_m^{Q\omega}}}, \frac{\sqrt{\rho_{M}^P}}{\sqrt{\rho_{m}^P}}\right)$ and $\xi^*\stackrel{\triangle}{=}36\Vert K^c\Vert^2(N+1)^3h_{\max}^2\max\left(\frac{\sqrt{\rho_{M}^P}}{\sqrt{\rho_{m}^P}},\frac{\rho_{M}^P}{\rho_m^{Q\omega}}, \frac{\rho_M^{Q\omega}}{\rho_{m}^P} \right)$,  we obtain
\begin{align}
  \frac{d}{dt}\left(\xi^{\frac{3}{2}}\sqrt{{\bar V}_c}+\theta^q\sqrt{{\bar V}_o}\right)
&\le -\frac{\xi^{\frac{3}{2}}\lambda}{4}\sqrt{{\bar V}_c}-\frac{\theta^{q+1}}{4}\sqrt{{\bar V}_o}+ k_5\theta^{q}\int_{t-\tau_M}^t\sqrt{\bar V_o}ds+ k_6\theta^{q}\int_{t-\tau_M}^t\sqrt{\bar V_c}ds\nonumber\\
 &\quad +k_3\left(\sum_{k=1}^q\xi^{\frac{3}{2}}\lambda^{q-k+1}\delta_{\varepsilon}^k\right)+k_9 \left(\sum_{k=1}^q\theta^{q-k+1}\delta_{\varepsilon}^k\right) +\theta^{q+1} k_{10}\left(\delta_w+\tau_M\delta_{\varepsilon}^{1}\right).
\end{align}
since $\xi^*\ge \max\left\{(6k_2)^2,(5k_7)^2,24k_8\right\}$ and $\lambda^*\ge\max\left\{20k_1,24k_4\right\}$ .\\
Applying Lemma \ref{lemma_continuous_discrete_time} with $v_1^2=\sqrt{\bar V_c(\eta^c)}$, $v_2^2=\sqrt{\bar V_o(\eta^o)}$, $\gamma_1=\xi^{\frac{3}{2}}$, $\gamma_2=\theta^q$, $a_1=\frac{\xi^{\frac{3}{2}}\lambda}{4}$, $a_2=\frac{\theta^{q+1}}{4}$, $b_1=\theta^qk_6$, $b_2= \theta^{q}k_5$ and $k=\theta^{q+1} k_{10}\left(\delta_w+\tau_M\delta_{\varepsilon}^{1}\right)+k_3\left(\sum_{k=1}^q\xi^{\frac{3}{2}}\lambda^{q-k+1}\delta_{\varepsilon}^k\right)+k_9 \left(\sum_{k=1}^q\theta^{q-k+1}\delta_{\varepsilon}^k\right)$, ensures that the following inequality holds true
\begin{align}
 \xi^{\frac{3}{2}}\sqrt{\bar V_c}+\theta^{q}\sqrt{\bar V_o}
 &\le \varsigma e^{-\frac{\lambda}{8}t}+\frac{\theta^{q+1}}{\lambda} 8 k_{10}\left(\delta_w+\tau_M\delta_{\varepsilon}^{1}\right)+\frac{8}{\lambda}k_3\left(\sum_{k=1}^q\xi^{\frac{3}{2}}\lambda^{q-k+1}\delta_{\varepsilon}^k\right)
 +\frac{8}{\lambda}k_9 \left(\sum_{k=1}^q\theta^{q-k+1}\delta_{\varepsilon}^k\right),\\
 & \le \varsigma e^{-\frac{\lambda}{8}t}+\xi\theta^{q} 8 k_{10}\left(\delta_w+\tau_M\delta_{\varepsilon}^{1}\right)+8k_3\left(\sum_{k=1}^q\xi^{\frac{3}{2}}\lambda^{q-k}\delta_{\varepsilon}^k\right)+8\xi  k_9 \left(\sum_{k=1}^q\theta^{q-k}\delta_{\varepsilon}^k\right),
\end{align}
with
\begin{equation}
 \varsigma=\xi^{\frac{3}{2}}\sqrt{\bar V_c(\eta^c(0))}+\theta^q\sqrt{\bar V_o(\eta^o(0))},
\end{equation}
provided that 
\begin{align}
 \tau_M<\min\left(\frac{(\sqrt{2}-1)\xi^{\frac{3}{2}}\lambda}{8\theta^qk_6},\frac{(\sqrt{2}-1)\theta}{8k_5},\frac{2\sqrt{2}}{\lambda},\frac{2\sqrt{2}}{\theta} \right), \nonumber
\end{align}
which is verified if
\begin{equation}
 \tau_M<\frac{(\sqrt{2}-1)\min(\sqrt{\omega_{\min}},\sqrt{\rho_{\min}(P)})}{8\Vert K^o\Vert (N+1)^{\frac{3}{2}}h_{\max}\sqrt{\rho_{\max}(P)}\bar c(L_{\varphi}+\theta)}
\end{equation}
since $(L_{\varphi}+\theta)\ge\theta\ge\lambda$.\\
Then, using the inequalities
\begin{align}
 \Vert x_i-x_0\Vert\le& \frac{1}{\lambda}\Vert e_i\Vert \le\frac{1}{\lambda}\left\Vert \begin{bmatrix} e_1\\\vdots\\ e_N \end{bmatrix} \right\Vert
 \le\frac{1}{\lambda} \Vert \eta_c\Vert,
 \le \frac{1}{\lambda\sqrt{\omega_{\min}\rho_{\min}(Q)}}\sqrt{\bar V_c(\eta^c)},\\
 \sqrt{\bar V_c(\eta^c)}&\le\lambda^q\sqrt{\omega_{\max}\rho_{\max}(Q)}\sum_{i=1}^N\Vert x_i-x_0\Vert,\\
 \sqrt{\bar V_o(\eta^o)}&\le\sqrt{\rho_{\max}(P)}\sum_{i=1}^N\sum_{j=0}^N\Vert\tilde x_{i,j}\Vert,
\end{align}
leads to
\begin{align}
 \Vert x_i-x_0\Vert&\le \theta^{q-1}\chi_1 e^{-\frac{\lambda}{8}t}+\lambda^{-1}\theta^{q}\chi_2\left(\delta_w+\tau_M \delta_{\varepsilon}^1\right)+\lambda^{-1}\chi_3\left(\sum_{k=1}^q\theta^{q-k}\delta_{\varepsilon}^k\right),
\end{align}
since $\theta\ge\lambda\ge1$ with $\chi_1=\frac{\sqrt{\omega_{\max}\rho_{\max}(Q)}}{\sqrt{\omega_{\min}\rho_{\min}(Q)}}\sum_{i=1}^N\Vert x_i(0)-x_0(0)\Vert+\frac{\sqrt{\rho_{\max}(P)}}{\sqrt{\omega_{\min}\rho_{\min}(Q)}}\sum_{i=1}^N\sum_{j=0}^N\Vert \hat x_{i,j}(0)-x_j(0)\Vert$, $\chi_2=\frac{8k_{10}}{\sqrt{\omega_{\min}\rho_{\min}(Q)}}$, $\chi_3=\frac{8(k_3+k_9)}{\sqrt{\omega_{\min}\rho_{\min}(Q)}}$.
\end{proof}

\section{Example}\label{section_example}
In this section, we present a numerical example, consisting of systems whose dynamics are the same as a Chua's oscillator, in order to illustrate the performances of the proposed protocol. Indeed, as in \cite{Song2010}, let us consider the MAS where each agent's dynamics are given by
\begin{equation}\label{eqn_mas_sys_example}
 \begin{cases}
  \dot x_i^{(1)}(t)=x_i^{(2)}(t),\\
  \dot x_i^{(2)}(t)=f\left(t,x_i^{(1)}(t),x_i^{(2)}(t)\right)+u_i(t)+\varepsilon_i^{(2)}(t),\\
  y_i=x_i^{(1)}+w_i,\quad i=1,\dots,10,
 \end{cases}
\end{equation}
where $x_i^{(1)}\stackrel{\triangle}{=}\begin{pmatrix}x_i^{(1,1)}&x_i^{(1,2)}&x_i^{(1,3)}\end{pmatrix}^T\in\mathbb R^3,x_i^{(2)}\stackrel{\triangle}{=}\begin{pmatrix}x_i^{(2,1)}&x_i^{(2,2)}&x_i^{(2,3)}\end{pmatrix}^T\in\mathbb R^3$ are the position and velocity of agent $i$, respectively, $u_i,y_i,\varepsilon_i^{(2)},w_i\in\mathbb R^3$ are the input, output, uncertainty and noise of each agent, and the nonlinear function $f$ is given by
\begin{equation*}
 f\left(t,x_i^{(1)},x_i^{(2)}\right)=\begin{pmatrix}
   \alpha\left(x_i^{(2,2)}-x_i^{(2,1)}-h\left(x_i^{(2,1)}\right)\right)\\x_i^{(2,1)}-x_i^{(2,2)}+x_i^{(2,3)}\\-\beta x_i^{(2,2)}-\gamma x_i^{(2,3)}-\beta\epsilon\sin\left(\omega x_i^{(1,1)}\right)
  \end{pmatrix},
\end{equation*}
where $\alpha=10$, $\beta=19.53$, $\gamma=0.1636$, $\epsilon=0.2$, $\omega=0.5$ and $h$ is a piece-wise linear function given by $h\left(x_i^{(2,1)}\right)=\frac{a-b}{2}\left(\left\vert x_i^{(2,1)}+1\right\vert-\left\vert x_i^{(2,1)}-1\right\vert\right)$ with parameters $a=-1.4325$ and $b=-0.7831$.\\
The communication graph of the MAS system (\ref{eqn_mas_sys_example}) is described on Figure \ref{fig_graph_example} and one assumes that only agents $3$ and $5$ receive the measured position of the leader which is referenced as agent $0$. 
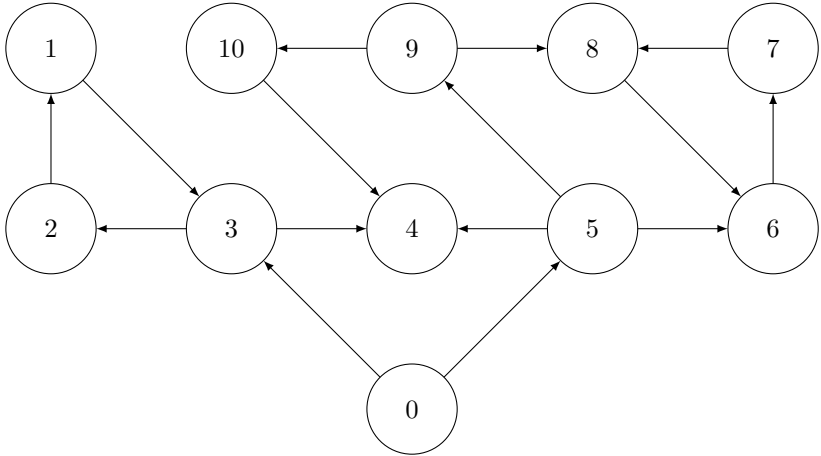
\begin{figure}[!ht]
 \begin{center}
\begin{tikzpicture}[scale=1.2]
 \draw (0,2) circle (0.5);
 \draw (0,2) node{$1$};
 \draw (0,0) circle (0.5);
 \draw (0,0) node{$2$};
   \draw (2,0) circle (0.5);
 \draw (2,0) node{$3$};
  \draw (4,0) circle (0.5);
 \draw (4,0) node{$4$};
   \draw (6,0) circle (0.5);
 \draw (6,0) node{$5$};
   \draw (8,0) circle (0.5);
 \draw (8,0) node{$6$};
  \draw (2,2) circle (0.5);
  \draw (2,2) node{$10$};
  \draw (4,2) circle (0.5);
 \draw (4,2) node{$9$};
   \draw (6,2) circle (0.5);
 \draw (6,2) node{$8$};
   \draw (8,2) circle (0.5);
 \draw (8,2) node{$7$};
 \draw (4,-2) circle (0.5);
 \draw (4,-2) node{$0$};
  \draw[>=latex,->](0,0.5)--(0,1.5);
  \draw[>=latex,->](1.5,0)--(0.5,0);
  \draw[>=latex,->] ({0.5*cos(-45)},{2+0.5*sin(-45)})--({2+0.5*cos(135)},{0.5*sin(135)});
  \draw[>=latex,->](2.5,0)--(3.5,0);
  \draw[>=latex,->](5.5,0)--(4.5,0);
  \draw[>=latex,->](6.5,0)--(7.5,0);
  \draw[>=latex,->](3.5,2)--(2.5,2);%
  \draw[>=latex,->](4.5,2)--(5.5,2);%
  \draw[>=latex,->](7.5,2)--(6.5,2);%
  \draw[>=latex,->] ({6+0.5*cos(-45)},{2+0.5*sin(-45)})--({8+0.5*cos(135)},{0.5*sin(135)});%
   \draw[>=latex,->] ({2+0.5*cos(-45)},{2+0.5*sin(-45)})--({4+0.5*cos(135)},{0.5*sin(135)});%
   \draw[>=latex,->] ({6+0.5*cos(135)},{0.5*sin(135)})--({4+0.5*cos(-45)},{2+0.5*sin(-45)});
  \draw[>=latex,->](8,0.5)--(8,1.5);%
  \draw[>=latex,->]({4+0.5*cos(135)},{-2+0.5*sin(135)})--({2+0.5*cos(-45)},{0.5*sin(-45)});
  \draw[>=latex,->]({4+0.5*cos(45)},{-2+0.5*sin(45)})--({6+0.5*cos(225)},{0.5*sin(225)});
\end{tikzpicture}
 \end{center}
 \caption{Communication graph of the MAS}
 \label{fig_graph_example}
\end{figure}
The leader-following consensus protocol (\ref{eqn_protocol_consensus_control})-(\ref{eqn_protocol_consensus_observer}) has been implemented in Matlab for minimum and maximum bound on the sampling periods equal to $\tau_m=0.02s$ and $\tau_M=0.04s$ respectively. The sampling periods have been set following a uniform distribution on $[\tau_m,\tau_M]$ independently for each edge. The first sampling periods corresponding to the transmission of the output of agent $2$ to agent $1$ are reported on figure \ref{fig_sampling_periods_agent_1_agent_2}.\\
\begin{figure}[!ht]
\begin{center}
    \includegraphics[width=10cm]{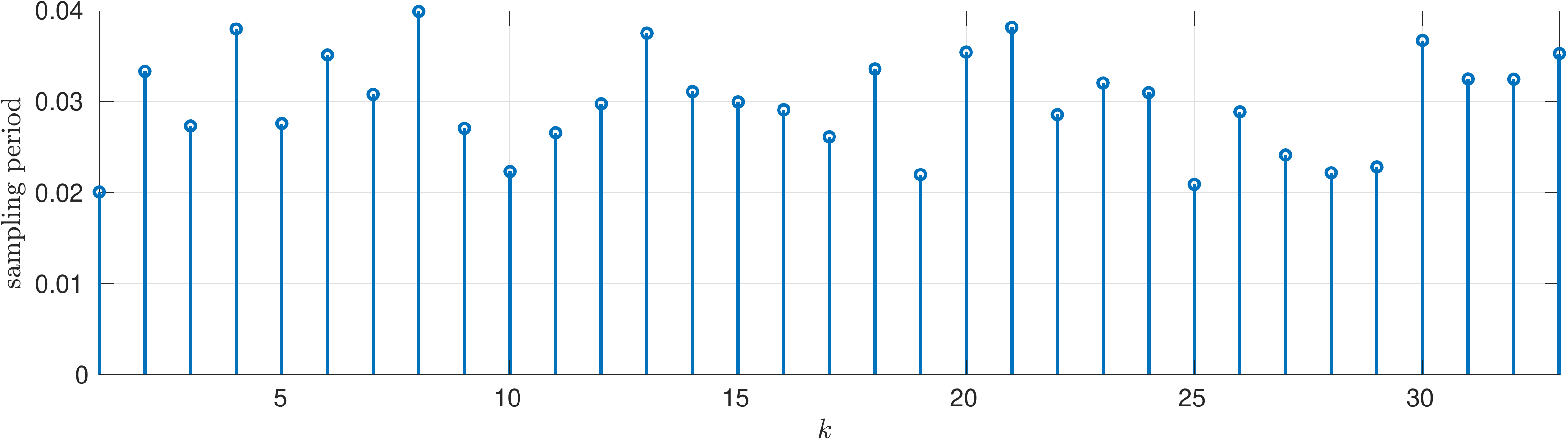}
\end{center}
\caption{Sampling periods for the transmission from agent $2$ to agent $1$.}
\label{fig_sampling_periods_agent_1_agent_2}
\end{figure}
The tuning parameters $\bar c,\theta,\lambda$ have been chosen by trial and error and taken equal as $\bar c=1$, $\theta=20$ and $\lambda=2$.\\
A first simulation has been conducted with no uncertainty on the dynamics and no noise on the outputs, that is $\varepsilon_i^{(2)}=0$ and $w_i=0$. The position of the different agents are reported on Figure \ref{fig_simulation_results_position}a)c)e). Furthermore, the estimation error of agent $2$ by agent $1$ is depicted in Figure \ref{figure_estimation_example}a)c). The position tracking mean error $\frac{1}{N}\sum_{i=1}^N\left\Vert x_i^{(1)}(t)-x_0^{(1)}(t)\right\Vert$ is reported on figure \ref{figure_mean_tracking_position_error}a). As expected by the theory, both the tracking errors and the observation errors go to zero exponentially.\\
It is worth to be noted that only the outputs, namely the position of the agents, are transmitted through the network and according to the topology described on Figure \ref{fig_graph_example}. The velocities and inputs are not transmitted and then unknown by neighbors agents. Furthermore, each agent transmits its output at time instants independently from its neighbors.\\
Another simulation has been conducted, in the same conditions but with a non zero uncertainty only on the leader's dynamics, corresponding to an unknown leader input (that is $\varepsilon_0^{(2)}\not\equiv0$ and $\varepsilon_i^{(2)}\equiv 0$ for $i=1,\dots,10$) and noise on the transmitted outputs. More precisely, the uncertainty on the leader is given by $\varepsilon_0^{(2)}(t)=\begin{pmatrix} \cos(t)&\cos(2t)&\cos(3t)\end{pmatrix}^T$ and the noise on the outputs of the agents and the leader $w_i$ are centered white noise with variance equal to $0.1$. The second component of the noisy and non noisy outputs of the leader are reported on figure \ref{figure_noisy_output}. The positions of the agents are reported on figure \ref{fig_simulation_results_position}b)d)f). The estimation error of agent $2$ by agent $1$ is reported on Figure \ref{figure_estimation_example}b)d). The position tracking mean error $\frac{1}{N}\sum_{i=1}^N\left\Vert x_i^{(1)}(t)-x_0^{(1)}(t)\right\Vert$ is reported on figure \ref{figure_mean_tracking_position_error}b). While the position tracking error does not converge exactly to zero, the effect of the uncertainties has been lowered by taking $\theta$ and $\lambda$ sufficiently high. It should be noted that in the case of noisy measurements, very high values of $\theta$ will lead to an amplification of the noise in the reconstructed state. Thus, a trade-off on the value of $\theta$ and $\lambda$ has to be done. Indeed, sufficiently high values of these parameters have to be considered to attenuate the effect of the uncertainties but not too high such that the noise is not amplified too much. Nevertheless, despite uncertainties and noise, the proposed leader following consensus protocol still performs well as illustrated in the simulation.
\begin{figure}
 \begin{center}
  \includegraphics[width=8cm]{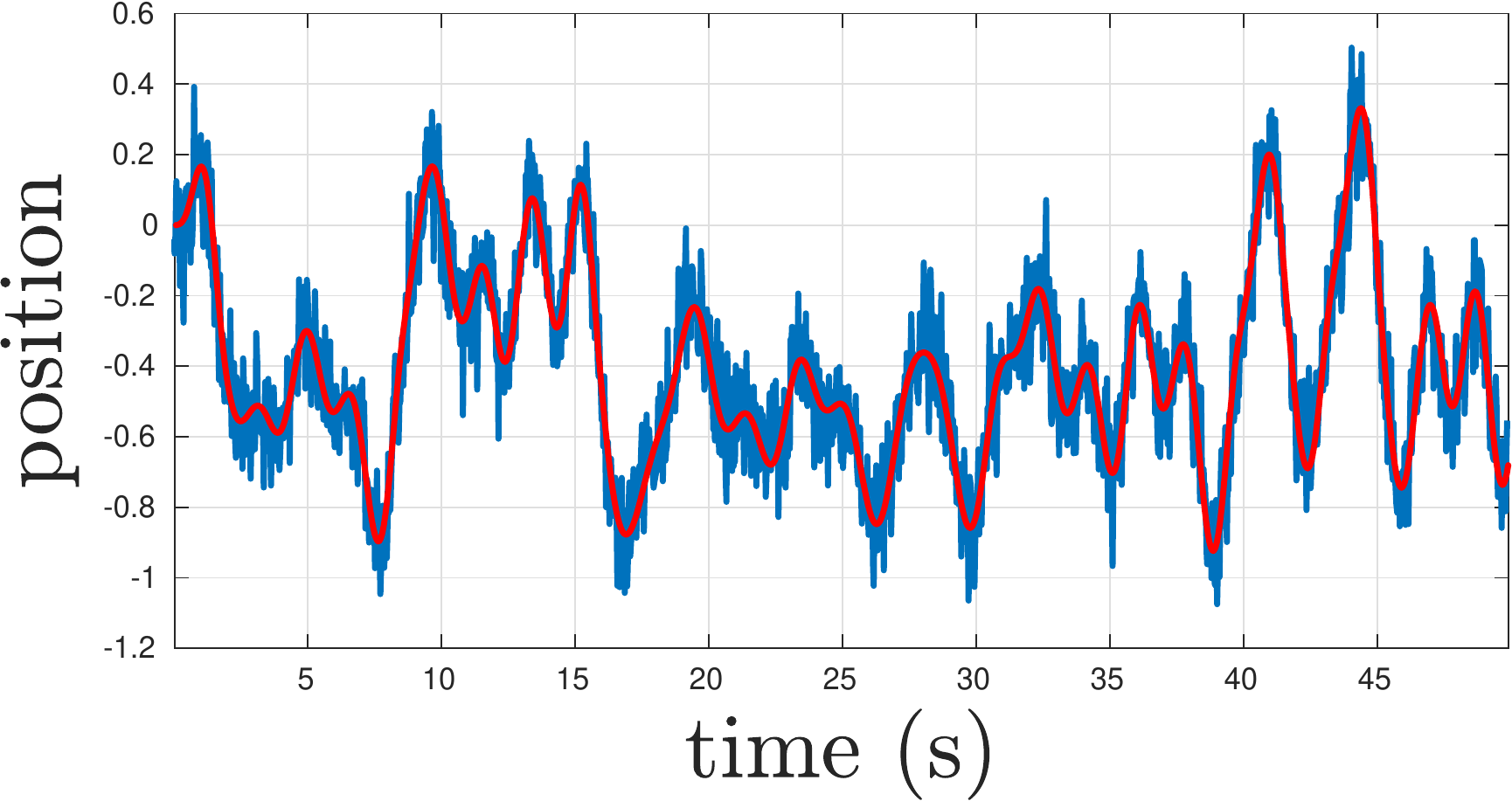}
 \end{center}
\caption{Second component of $y_0$ with noise (red) and without noise (blue)}
\label{figure_noisy_output}
\end{figure}
\begin{figure}[!ht] 
\begin{center}
\begin{tabular}{cc}
 \subfloat[Positions $x_i^{(1,1)}$, $i=0,\dots,10$]{\includegraphics[width=7cm]{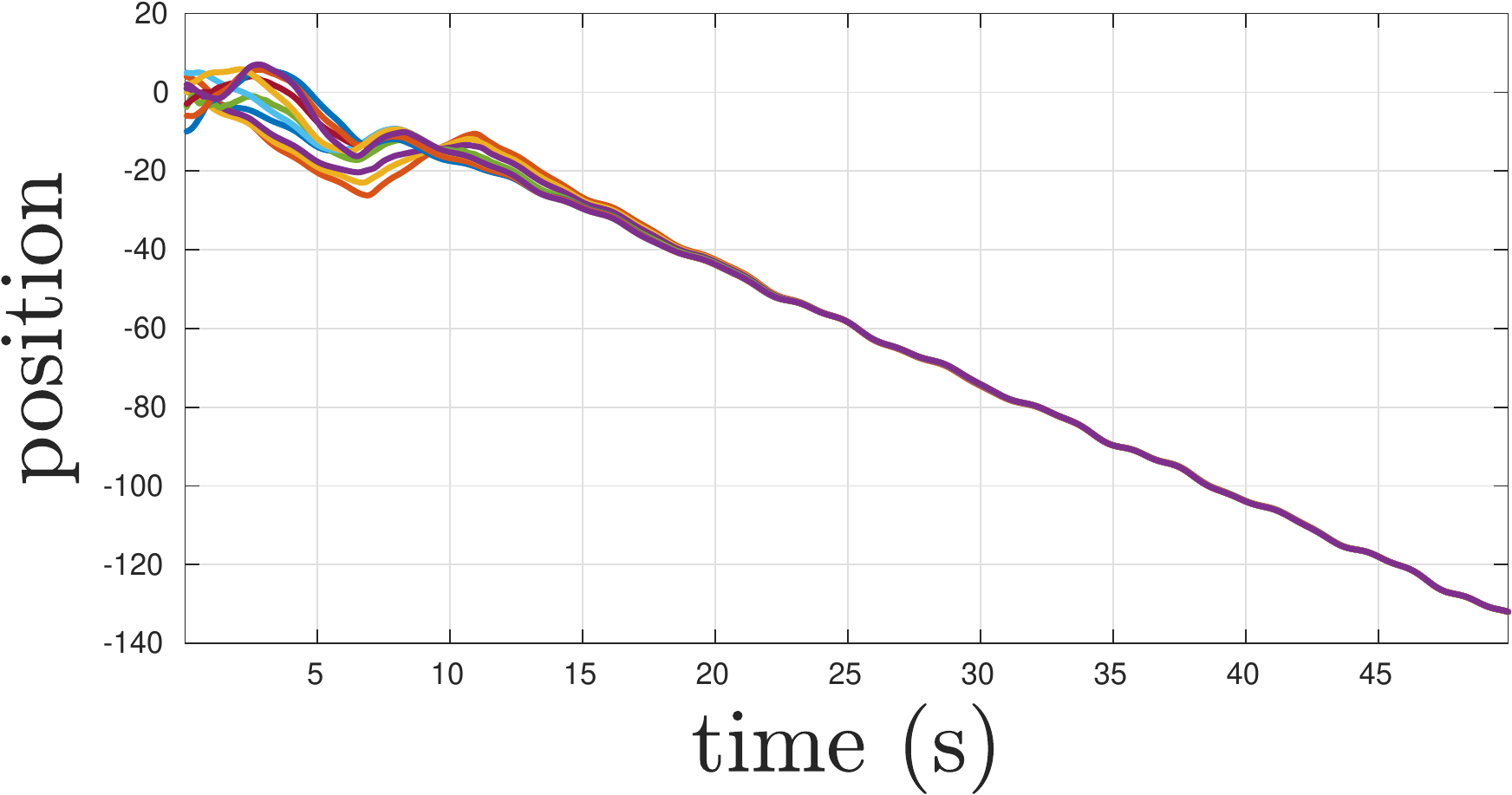}}
 &
  \subfloat[Positions $x_i^{(1,1)}$, $i=0,\dots,10$]{\includegraphics[width=7cm]{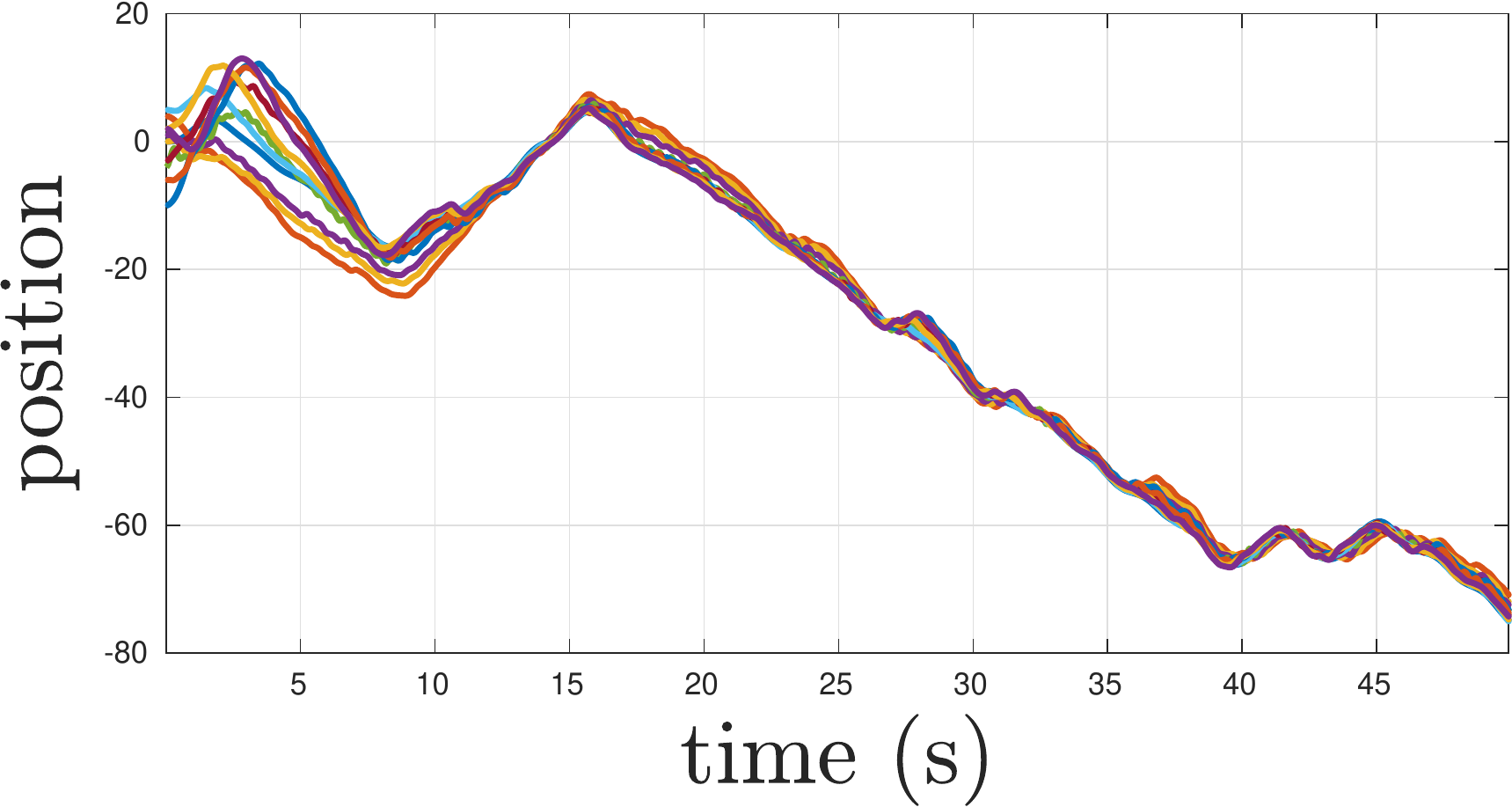}}
\\
 \subfloat[Positions $x_i^{(1,2)}$, $i=0,\dots,10$]{\includegraphics[width=7cm]{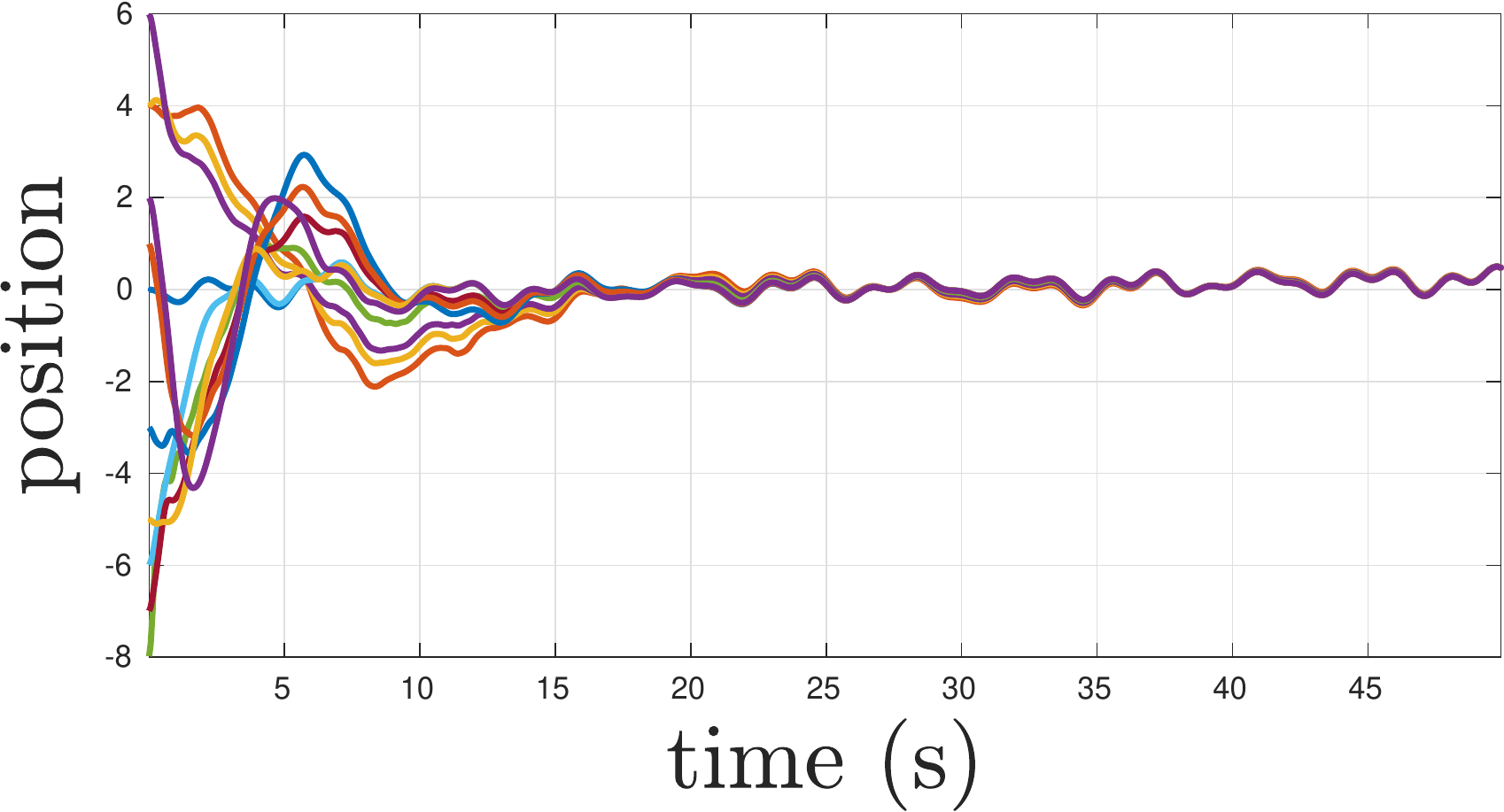}}
 &
  \subfloat[Positions $x_i^{(1,2)}$, $i=0,\dots,10$]{\includegraphics[width=7cm]{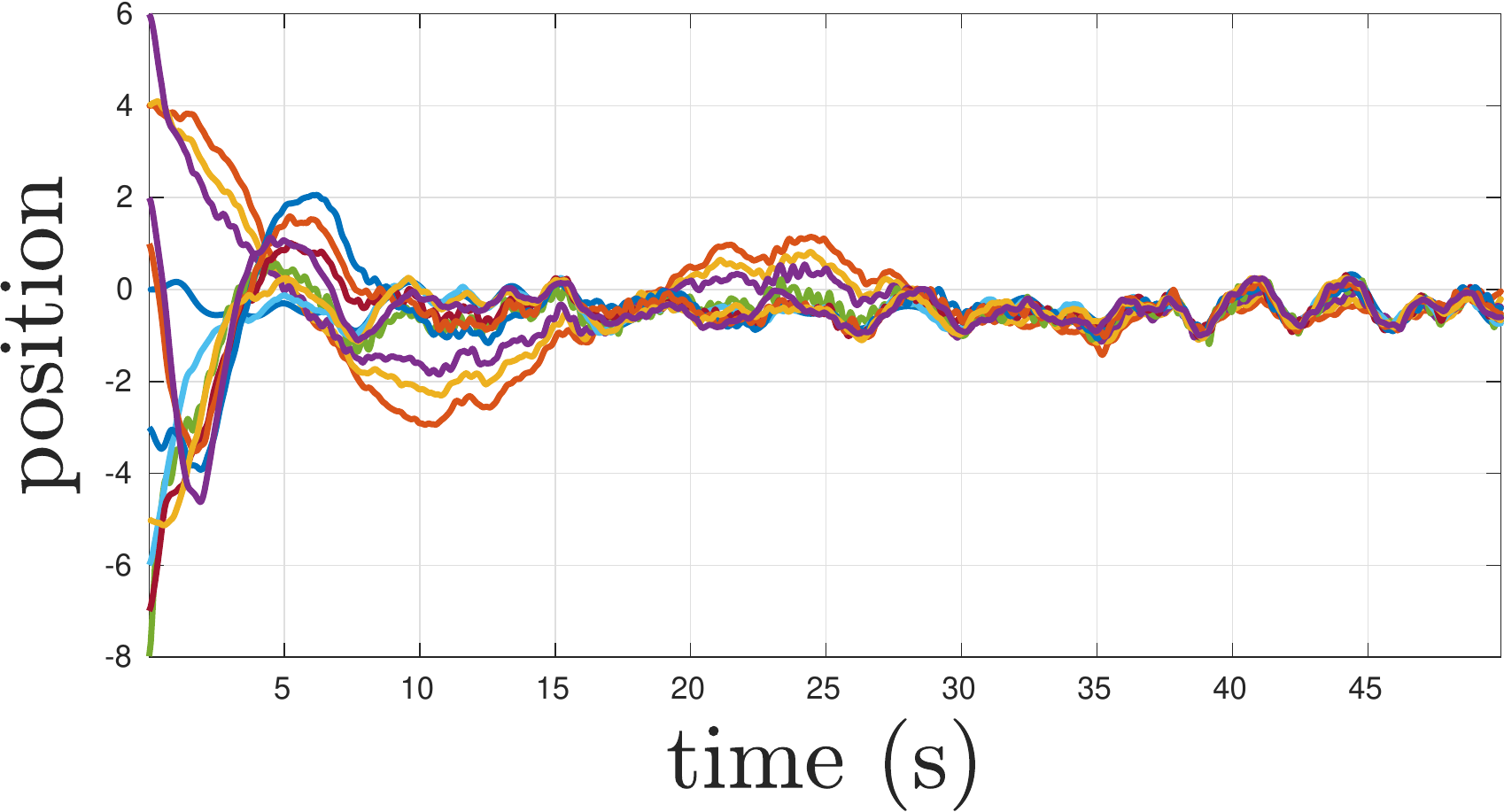}}
   \\
    \subfloat[Positions $x_i^{(1,3)}$, $i=0,\dots,10$]{\includegraphics[width=7cm]{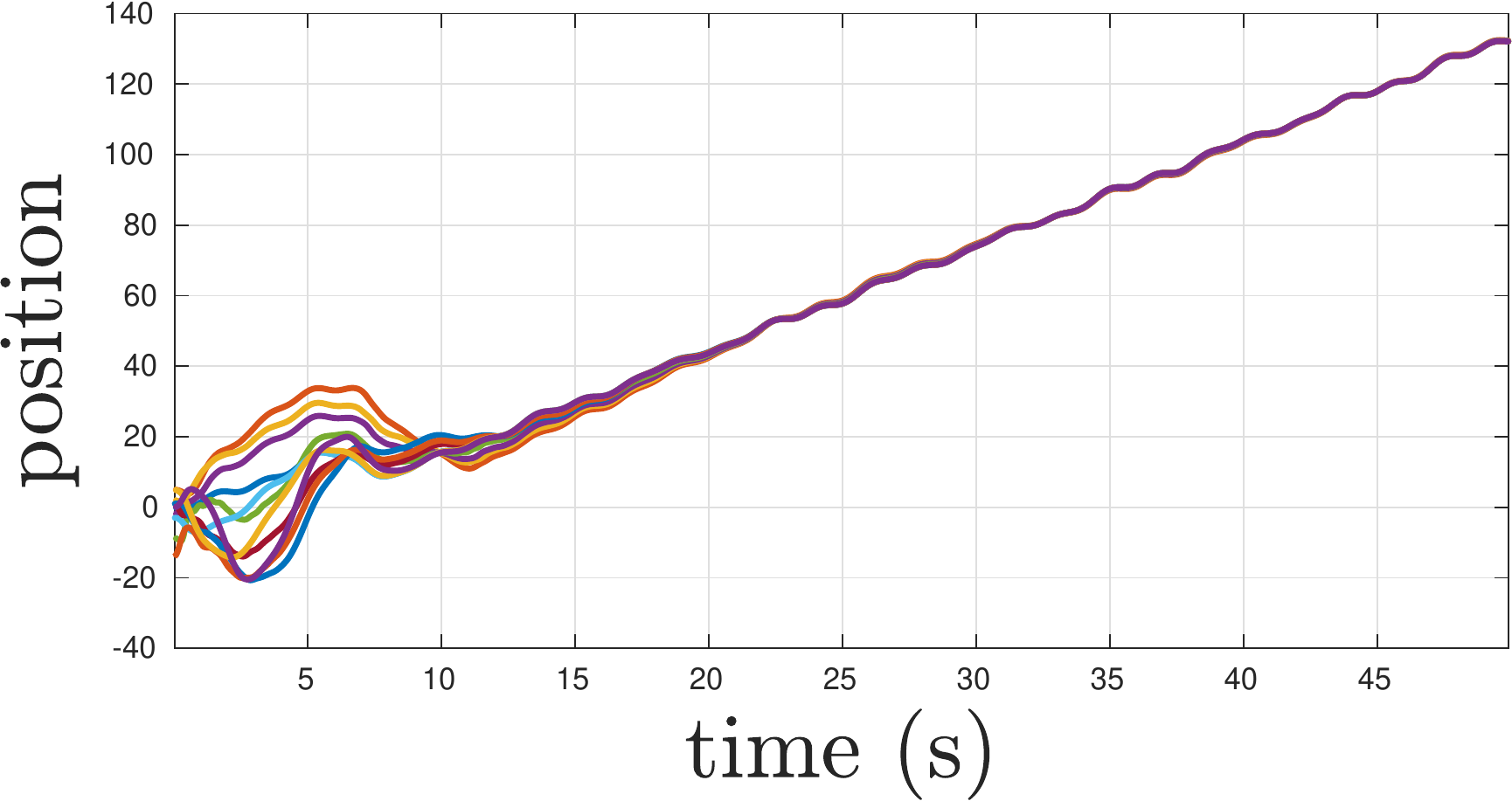}}
    &
     \subfloat[Positions $x_i^{(1,3)}$, $i=0,\dots,10$]{\includegraphics[width=7cm]{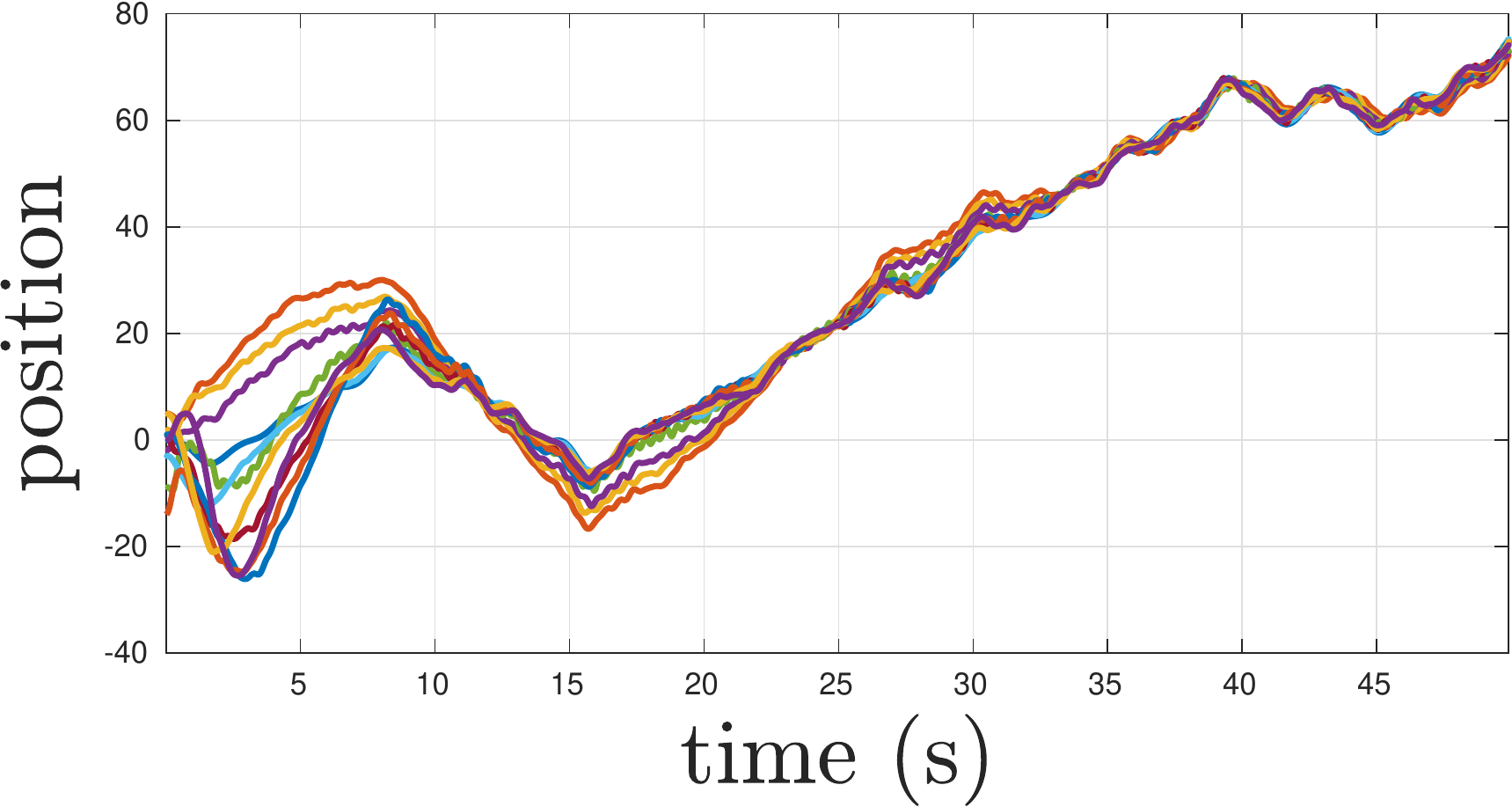}}
   \end{tabular}
  \end{center}
  \caption{Positions of the agents for $\tau_m=20ms$, $\tau_M=40ms$, $\theta=20$ and $\lambda=2$}
  \label{fig_simulation_results_position}
\end{figure}

\begin{figure}
\begin{center}
\begin{tabular}{cc}
\subfloat[Position error \tiny{$\left\Vert\hat x_{1,2}^{(1)}-x_2^{(1)}\right\Vert$}]{\includegraphics[width=7cm]{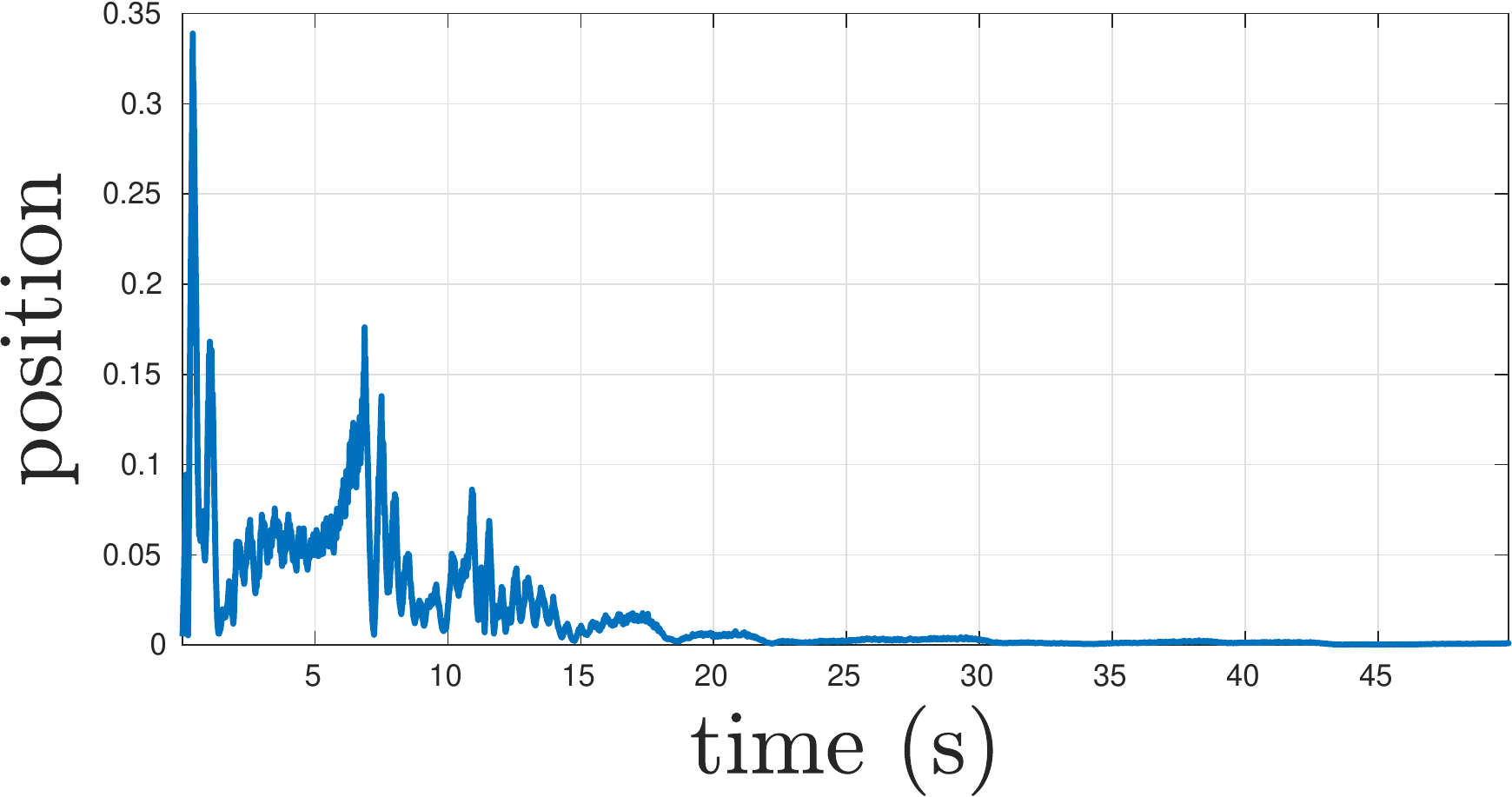}}
&
\subfloat[Position error \tiny{$\left\Vert\hat x_{1,2}^{(1)}-x_2^{(1)}\right\Vert$}]{\includegraphics[width=7cm]{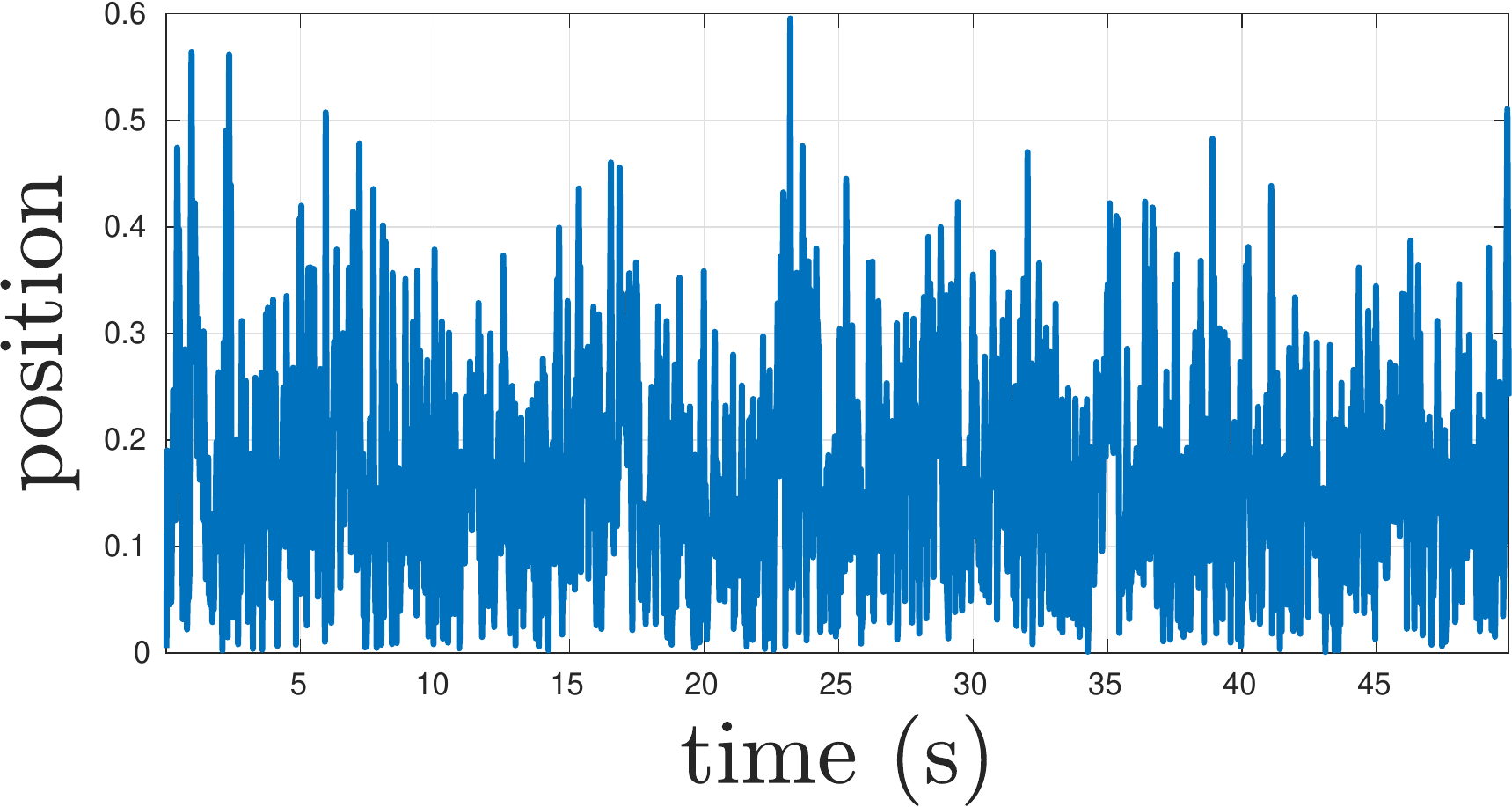}}
 \\
 \subfloat[Velocity error \tiny{$\left\Vert\hat x_{1,2}^{(2)}-x_2^{(2)}\right\Vert$}]{
 \includegraphics[width=7cm]{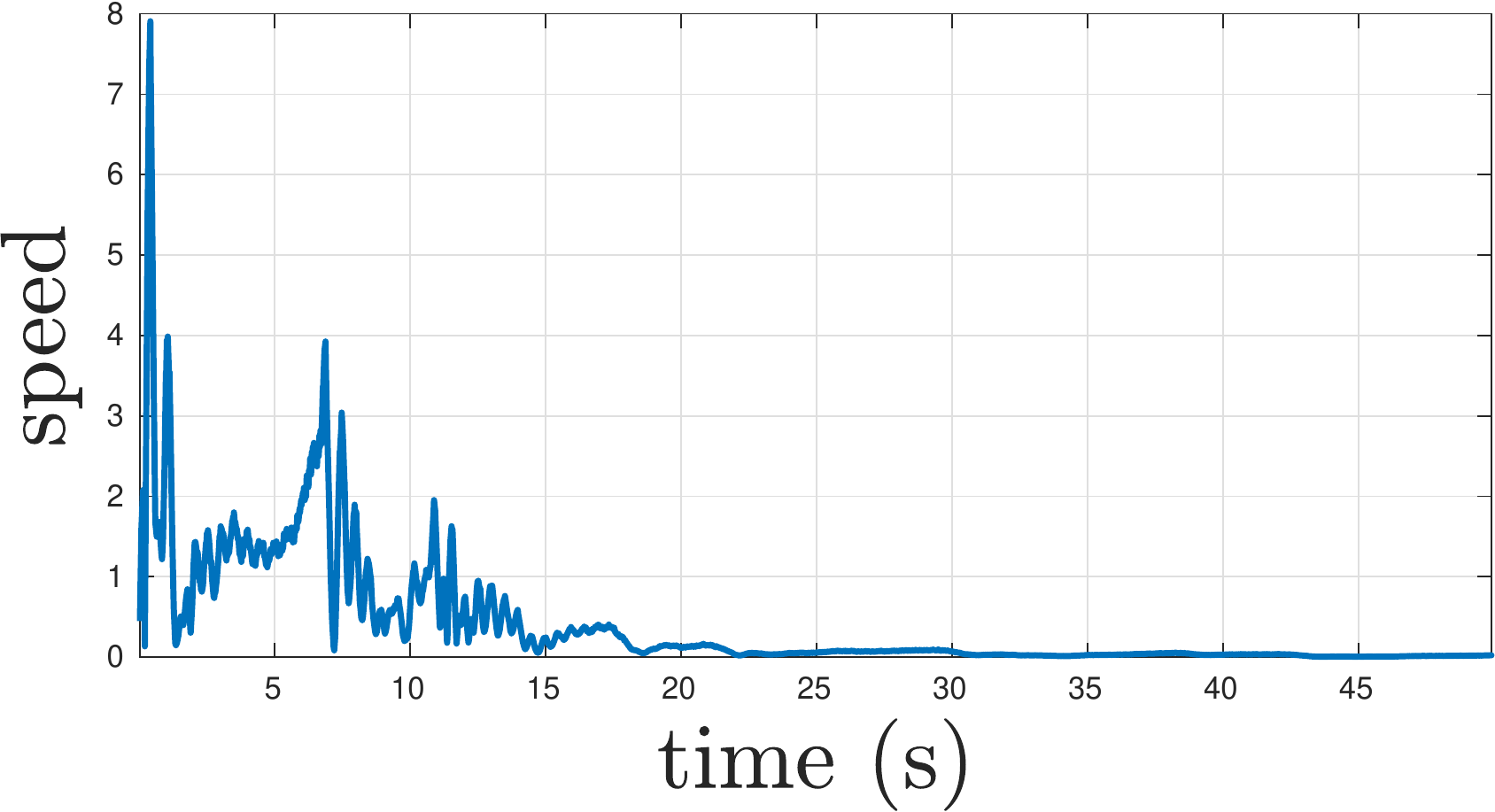}}
 &
  \subfloat[Velocity error \tiny{$\left\Vert\hat x_{1,2}^{(2)}-x_2^{(2)}\right\Vert$}]{
 \includegraphics[width=7cm]{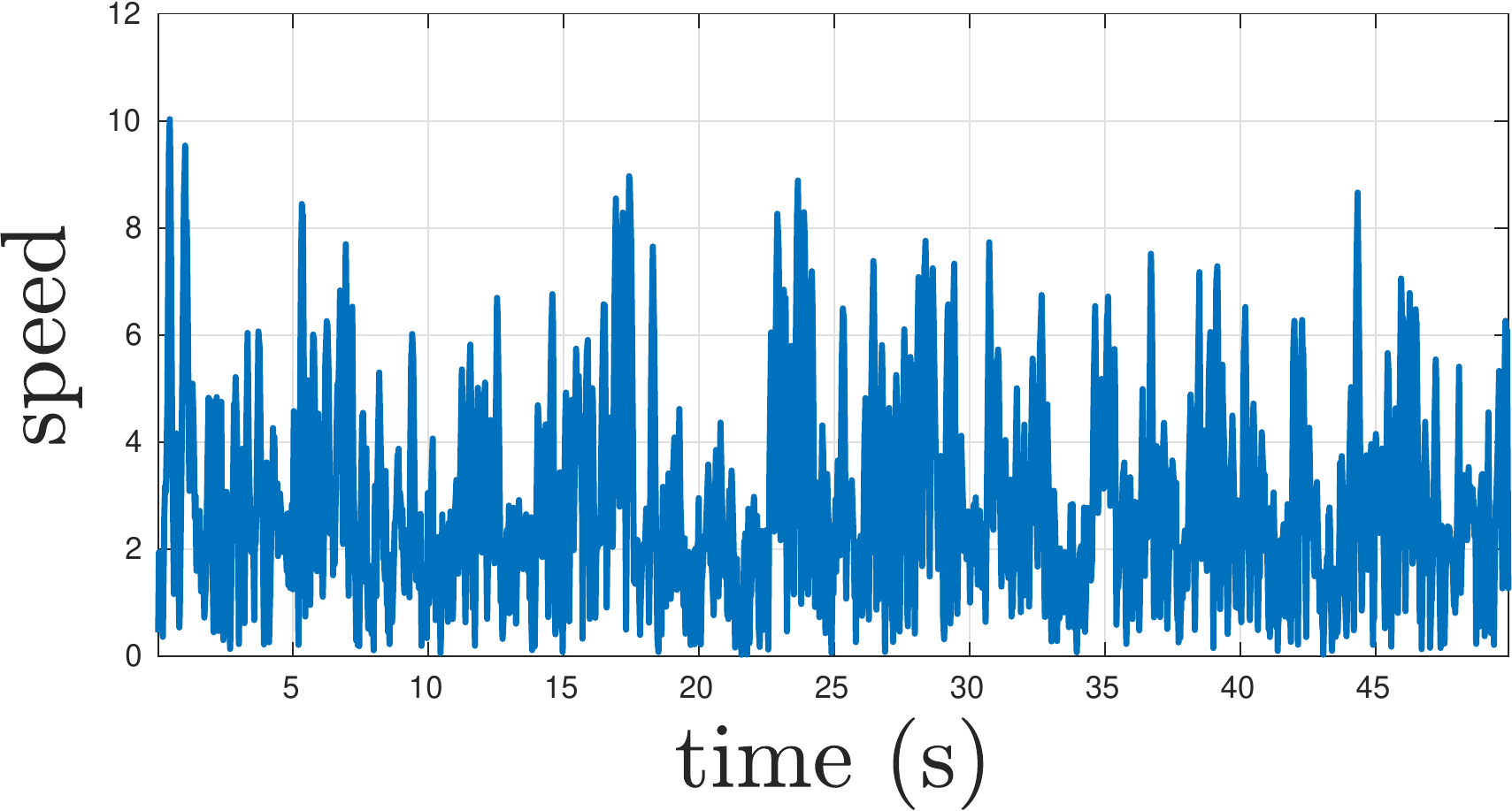}}
 \end{tabular}
 \end{center}
 \caption{Estimation error of the state of agent $2$ by agent $1$}
 \label{figure_estimation_example}
\end{figure}
\begin{figure}
\begin{center}
 \begin{tabular}{cc}
  \subfloat[Simulation without noise and uncertainty]{\includegraphics[width=7cm]{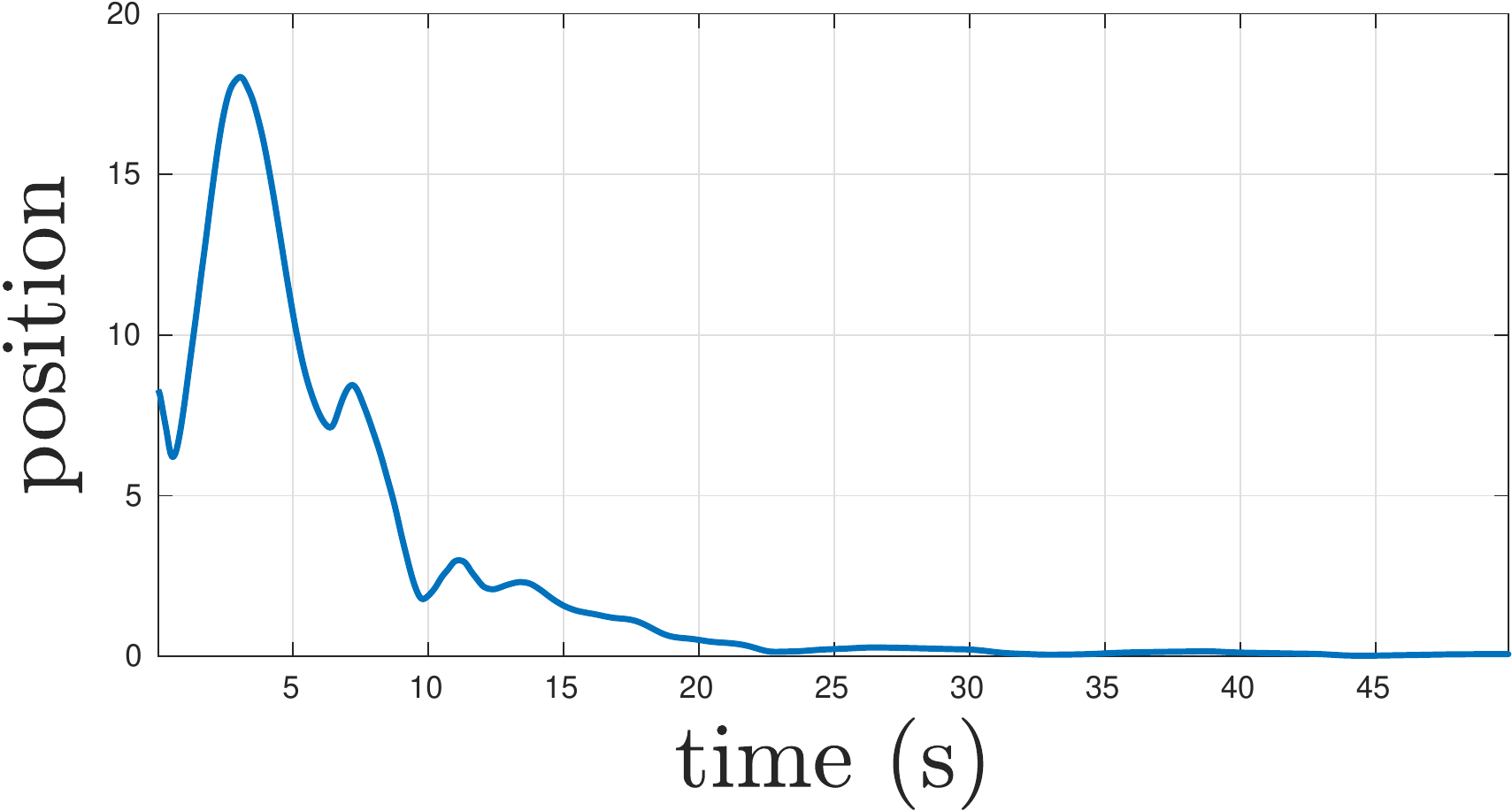}}
  &
   \subfloat[Simulation with noise and uncertainty]{\includegraphics[width=7cm]{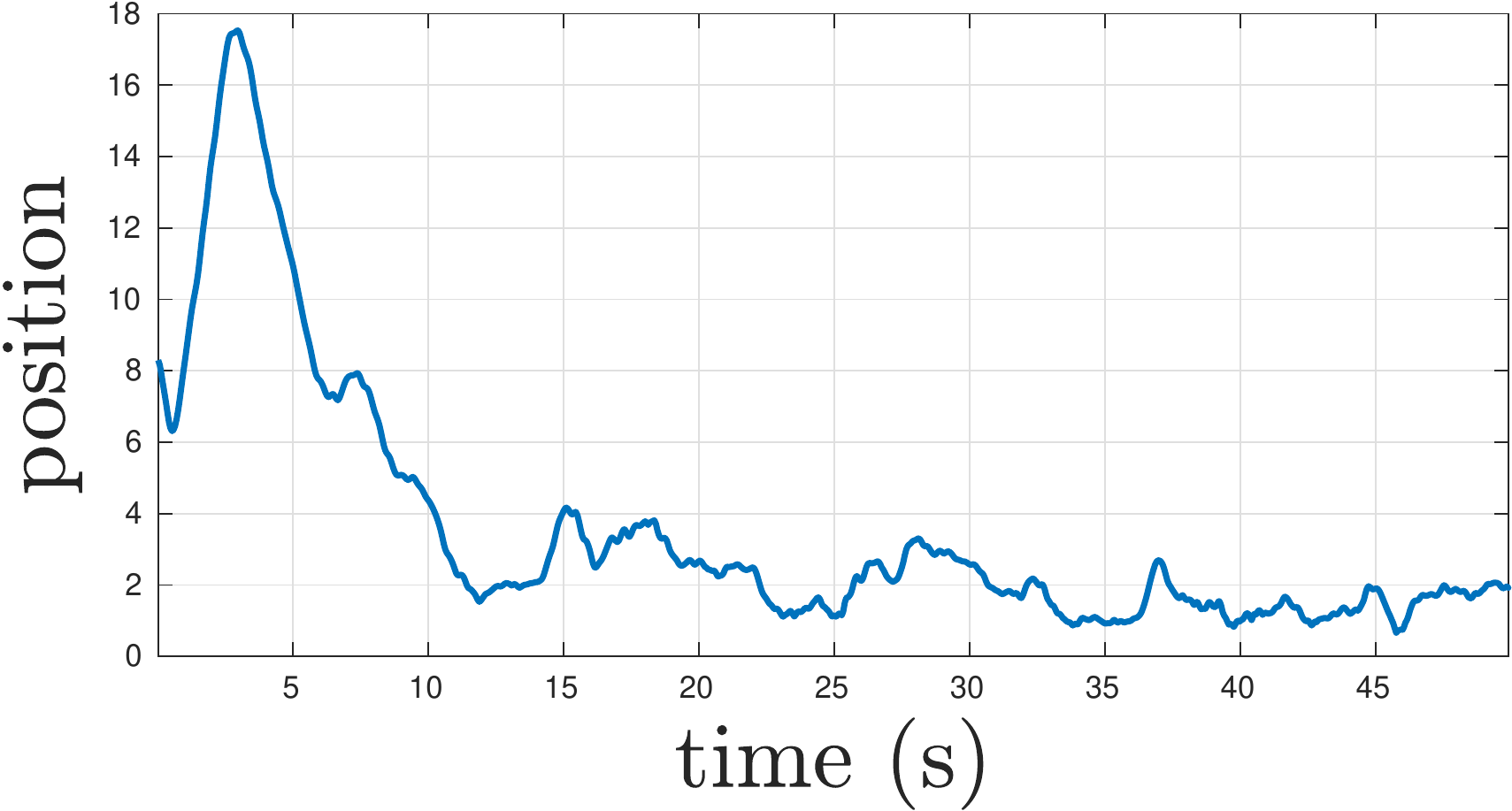}}
 \end{tabular}
 \end{center}
\caption{Position tracking mean error \tiny{$\frac{1}{N}\sum_{i=1}^N \left\Vert x_i^{(1)}-x_0^{(1)}\right\Vert$}}
\label{figure_mean_tracking_position_error}
\end{figure}

\section{Conclusion}\label{section_conclusion}
 The problem of leader-following consensus for a class of nonlinear systems which can only transmit their outputs at discrete aperiodic and asynchronous instants has been considered in this paper. A consensus protocol has been proposed based on a continuous-discrete time observer that reconstruct the states of the neighbors in continuous time, from sampled measurements only, and a continuous control law. It has been shown theoretically that if the tuning parameters fulfill some sufficient conditions then the convergence of the MAS with the proposed protocol is ensured under directed topology. Furthermore, in case of bounded uncertainties on the dynamics and bounded noise on the output, an exponential practical consensus is guaranteed. The performances of the approach have been illustrated with simulations on a MAS whose agents dynamics are given by a Chua's oscillator.
\bibliographystyle{plain}
\bibliography{bib_consensus_nonlinear.bib}
\end{document}